\documentclass[fleqn,reqno,10pt,a4paper,final]{amsart}

\usepackage{amsmath}
\usepackage{amssymb}
\usepackage{amsthm}
\usepackage{amscd}
\usepackage[ansinew]{inputenc}
\usepackage{cite}
\usepackage{bbm}
\usepackage{color}
\usepackage[english=american]{csquotes}
\usepackage[final]{graphicx}

\numberwithin{equation}{section}

\newtheoremstyle{thmlemcorr}{10pt}{10pt}{\itshape}{}{\bfseries}{.}{10pt}{{\thmname{#1}\thmnumber{ #2}\thmnote{ (#3)}}}
\newtheoremstyle{thmlemcorr*}{10pt}{10pt}{\itshape}{}{\bfseries}{.}\newline{{\thmname{#1}\thmnumber{ #2}\thmnote{ (#3)}}}
\newtheoremstyle{defi}{10pt}{10pt}{\itshape}{}{\bfseries}{.}{10pt}{{\thmname{#1}\thmnumber{ #2}\thmnote{ (#3)}}}
\newtheoremstyle{remexample}{10pt}{10pt}{}{}{\bfseries}{.}{10pt}{{\thmname{#1}\thmnumber{ #2}\thmnote{ (#3)}}}
\newtheoremstyle{ass}{10pt}{10pt}{}{}{\bfseries}{.}{10pt}{{\thmname{#1}\thmnumber{ A#2}\thmnote{ (#3)}}}

\theoremstyle{thmlemcorr}
\newtheorem{theorem}{Theorem}
\numberwithin{theorem}{section}
\newtheorem{lemma}[theorem]{Lemma}
\newtheorem{corollary}[theorem]{Corollary}
\newtheorem{proposition}[theorem]{Proposition}

\theoremstyle{thmlemcorr*}
\newtheorem{theorem*}{Theorem}
\newtheorem{lemma*}[theorem]{Lemma}
\newtheorem{corollary*}[theorem]{Corollary}
\newtheorem{proposition*}[theorem]{Proposition}
\newtheorem{problem*}[theorem]{Problem}
\newtheorem{conjecture*}[theorem]{Conjecture}

\theoremstyle{defi}

\theoremstyle{remexample}
\newtheorem{remark}[theorem]{Remark}

\theoremstyle{ass}

\newcommand{\Crm}{\mathrm{C}}

\newcommand{\Lrm}{\mathrm{L}}

\newcommand{\Wrm}{\mathrm{W}}

\newcommand{\Fcal}{\mathcal{F}}

\newcommand{\Hcal}{\mathcal{H}}

\newcommand{\Lcal}{\mathcal{L}}

\newcommand{\Ebf}{\mathbf{E}}

\newcommand{\Mbf}{\mathbf{M}}

\newcommand{\Ybf}{\mathbf{Y}}

\renewcommand{\Bbb}{\mathbb{B}}

\newcommand{\Rbb}{\mathbb{R}}
\newcommand{\Sbb}{\mathbb{S}}

\DeclareMathOperator{\id}{id}

\DeclareMathOperator*{\wslim}{w*-lim}

\DeclareMathOperator{\dist}{dist}
\DeclareMathOperator{\rk}{rank}
\DeclareMathOperator{\supp}{supp}

\newcommand{\setn}[2]{\{\, #1 \ \ \textup{\textbf{:}}\ \ #2 \,\}}
\newcommand{\setb}[2]{\bigl\{\, #1 \ \ \textup{\textbf{:}}\ \ #2 \,\bigr\}}

\newcommand{\norm}[1]{\|#1\|}

\newcommand{\abs}[1]{|#1|}

\newcommand{\absn}[1]{|#1|}
\newcommand{\absb}[1]{\bigl|#1\bigr|}

\newcommand{\absBB}[1]{\biggl|#1\biggr|}
\newcommand{\absBBB}[1]{\Biggl|#1\Biggr|}
\newcommand{\floor}[1]{\lfloor #1 \rfloor}

\newcommand{\floorBB}[1]{\biggl\lfloor #1\biggr\rfloor}

\newcommand{\dpr}[1]{\langle #1 \rangle}	

\newcommand{\dprn}[1]{\langle #1 \rangle}
\newcommand{\dprb}[1]{\bigl\langle #1 \bigr\rangle}

\newcommand{\ddpr}[1]{\langle\!\langle #1 \rangle\!\rangle}

\newcommand{\ddprn}[1]{\langle\!\langle #1 \rangle\!\rangle}
\newcommand{\ddprb}[1]{\bigl\langle\hspace{-2.5pt}\bigl\langle #1 \bigr\rangle\hspace{-2.5pt}\bigr\rangle}
\newcommand{\ddprB}[1]{\Bigl\langle\!\!\Bigl\langle #1 \Bigr\rangle\!\!\Bigr\rangle}

\newcommand{\cl}[1]{\overline{#1}}
\newcommand{\di}{\mathrm{d}}
\newcommand{\dd}{\;\mathrm{d}}

\newcommand{\N}{\mathbb{N}}
\newcommand{\R}{\mathbb{R}}

\newcommand{\ONE}{\mathbbm{1}}

\newcommand{\toweakstar}{\overset{*}\rightharpoondown}

\newcommand{\todown}{\downarrow}

\newcommand{\toY}{\overset{\Ybf}{\to}}

\newcommand{\sbullet}{\begin{picture}(1,1)(-0.5,-2)\circle*{2}\end{picture}}
\newcommand{\frarg}{\,\sbullet\,}
\newcommand{\BV}{\mathrm{BV}}

\newcommand{\const}{\mathrm{const}}
\newcommand{\GYbf}{\mathbf{GY}}

\DeclareMathOperator{\Tan}{Tan}

\newcommand{\term}[1]{\textbf{#1}}

\newcommand{\proofstep}[1]{\textit{#1}}

\def\Xint#1{\mathchoice 
{\XXint\displaystyle\textstyle{#1}}%
{\XXint\textstyle\scriptstyle{#1}}%
{\XXint\scriptstyle\scriptscriptstyle{#1}}%
{\XXint\scriptscriptstyle\scriptscriptstyle{#1}}%
\!\int} 
\def\XXint#1#2#3{{\setbox0=\hbox{$#1{#2#3}{\int}$} 
\vcenter{\hbox{$#2#3$}}\kern-.5\wd0}} 
\def\dashint{\,\Xint-}

\newcommand{\restrict}{\begin{picture}(10,8)\put(2,0){\line(0,1){7}}\put(1.8,0){\line(1,0){7}}\end{picture}}

\renewcommand{\rho}{\varrho}
\renewcommand{\epsilon}{\varepsilon}
\renewcommand{\phi}{\varphi}

\begin{document}

\title[BV-Lower semicontinuity without Alberti's Theorem]{Lower semicontinuity and Young measures in BV without Alberti's Rank-One Theorem}

\author{Filip Rindler}
\address{Mathematical Institute, University of Oxford, 24--29 St Giles', Oxford OX1 3LB, United Kingdom}
\email{rindler@maths.ox.ac.uk}

\begin{abstract}
We give a new proof of sequential weak* lower semicontinuity in $\BV(\Omega;\R^m)$ for integral functionals of the form
\begin{align*}
  \Fcal(u) := &\int_\Omega f( x, \nabla u) \dd x
  + \int_\Omega f^\infty \Bigl( x, \frac{\di D^s u}{\di \abs{D^s u}} \Bigr) \dd \abs{D^s u} \\
  &\qquad + \int_{\partial\Omega} f^\infty \bigl( x, u \otimes n_{\Omega} \bigr) \dd \Hcal^{d-1},
  \qquad u \in \BV(\Omega;\R^m),
\end{align*}
where $f \colon \cl{\Omega} \times \R^{m \times d} \to \R$ is a quasiconvex Carath\'{e}odory integrand with linear growth at infinity, i.e.\ $\abs{f(x,A)} \leq M(1+\abs{A})$ for some $M \geq 0$, and such that the recession function $f^\infty(x,A) := \lim_{x'\to x,t\to\infty} t^{-1}f(x',tA)$ exists and is (jointly) continuous. In contrast to the classical proofs by Ambrosio \& Dal Maso [\textit{J.\ Funct.\ Anal.} 109 (1992), 76--97] and Fonseca \& M\"{u}ller [\textit{Arch.\ Ration.\ Mech.\ Anal.} 123 (1993), 1--49], we do not use Alberti's Rank-One Theorem [\textit{Proc.\ Roy.\ Soc.\ Edinburgh Sect.\ A} 123 (1993), 239--274], but a rigidity result for gradients. The proof is set in the framework of generalized Young measures and proceeds via establishing Jensen-type inequalities for regular and singular points of $Du$.

\vspace{8pt}



\vspace{8pt}

\noindent\textsc{Date:} \today.
\end{abstract}

\maketitle

\section{Introduction}

The aim of this work is to give a new proof of a classical lower semicontinuity theorem for integral functionals on the space $\BV(\Omega;\R^m)$ of vector-valued functions of bounded variation:

\begin{theorem} \label{thm:BV_lsc_teaser}
Let $\Omega \subset \R^d$ be a bounded Lipschitz domain with boundary unit inner normal $n_{\Omega} \colon \partial \Omega \to \Sbb^{d-1}$. Further, let $f \colon \cl{\Omega} \times \R^{m \times d} \to \R$ be a Carath\'{e}odory integrand with linear growth at infinity, i.e.\ $\abs{f(x,A)} \leq M(1+\abs{A})$ for some constant $M \geq 0$, that is quasiconvex in its second argument and for which the recession function
\[
  f^\infty(x,A) := \lim_{\substack{\!\!\!\! x' \to x \\ \; t \to \infty}}
    \frac{f(x',tA)}{t},  \qquad \text{$x \in \cl{\Omega}$, $A \in \R^{m \times d}$,}
\]
exists and is (jointly) continuous. Then, the functional
\begin{align*}
  \Fcal(u) := &\int_\Omega f( x, \nabla u) \dd x
    + \int_\Omega f^\infty \Bigl( x, \frac{\di D^s u}{\di \abs{D^s u}} \Bigr)
    \dd \abs{D^s u} \\
  &\qquad + \int_{\partial\Omega} f^\infty \bigl( x, u|_{\partial \Omega}
    \otimes n_{\Omega} \bigr) \dd \Hcal^{d-1},
    \qquad u \in \BV(\Omega;\R^m),
\end{align*}
where $u|_{\partial \Omega} \in \Lrm^1(\partial \Omega,\Hcal^{d-1};\R^m)$ is the inner boundary trace of $u$ on $\partial \Omega$, is sequentially lower semicontinuous with respect to weak* convergence in the space $\BV(\Omega;\R^m)$.
\end{theorem}

This result was first established by Ambrosio \& Dal Maso~\cite{AmbDal92RBVQ} and Fonseca \& M\"{u}ller~\cite{FonMul93RQFB}, also see~\cite{FonMul92QCIL}, which introduced the blow-up method employed in the proof, and~\cite{KriRin10RSIF} for the recent extension to signed integrands. Notice, however, that in the result above we need a stronger notion of recession function (not just the limes superior, but a proper limit); this phenomenon will be explained in Remark~\ref{rem:rec_funct}. An analogous lower semicontinuity theorem for symmetric-quasiconvex integral functionals with linear growth on the space BD of functions of bounded deformation has recently been proved by the author by employing a similar, yet more refined, strategy as in this paper, see~\cite{Rind10?LSIF}.

Traditionally, the proof of the above lower semicontinuity theorem crucially employs Alberti's Rank-One Theorem~\cite{Albe93ROPD}, which confirmed a conjecture of Ambrosio \& De Giorgi~\cite{AmbDeG88NTFC} and asserts that for a function $u \in \BV(\Omega;\R^m)$,
\[
  \rk \biggl( \frac{\di D^s u}{\di \abs{D^s u}}(x_0) \biggr) \leq 1
  \qquad\text{for $\abs{D^s u}$-almost every $x_0 \in \Omega$.}
\]
Here, $\frac{\di D^s u}{\di \abs{D^s u}}$ denotes the Radon--Nikod\'{y}m derivative of $D^s u$ with respect to the corresponding total variation measure $\abs{D^s u}$ (this particular density is sometimes called the \enquote{polar}). The proof of Alberti's Rank-One Theorem is rather involved, despite some recent efforts of simplification~\cite{DeLe09NARO} (there also is an announcement of a new proof in~\cite{AlCsPr05SNPA}). The main difficulty lies in the fact that it cannot be proved by the usual blow-up arguments, but requires a more sophisticated \enquote{decomposition} approach together with a clever use of the $\BV$-coarea formula.

This work will give a proof of the $\BV$-Lower Semicontinuity Theorem~\ref{thm:BV_lsc_teaser} that does not use Alberti's Rank-One Theorem, but instead combines the usual blow-up arguments with a rigidity lemma. \enquote{Rigidity} here means that all (exact) solutions to certain differential inclusions involving the gradient have additional structure. The decisive point is to realize that in the currently known proof of the lower semicontinuity theorem, Alberti's Theorem is employed only as a rigidity result: Blowing-up a function $u \in \BV(\Omega;\R^m)$ around a singular point $x_0 \in \Omega$ yields a $\BV$-function $v$ with constant polar function of the derivative, i.e.\
\[
  Dv = P(x_0) \abs{Dv}, \qquad\text{where}\qquad P(x_0) = \frac{\di D^s u}{\di \abs{D^s u}}(x_0).
\]
Alberti's Theorem now tells us that for $\abs{D^s u}$-almost every $x_0 \in \Omega$, $P(x_0) = a \otimes \xi$ for some $a \in \R^m$, $\xi \in \Sbb^{d-1}$, and hence we may infer that $v$ can be written as
\begin{equation} \label{eq:v_one_dir}
  v(x) = v_0 + \psi(x \cdot \xi)a \qquad\text{for some $\psi \in \BV(\R)$, $v_0 \in \R^m$.}
\end{equation}

The key observation in this paper is that a weaker statement can be proved much more easily: If $v \in \BV(C;\R^m)$, $C \subset \R^d$ an open convex set, satisfies
\[
  Dv = P \abs{Dv}  \qquad\text{where $P \in \R^{m \times d}$ with $\rk P \leq 1$},
\]
then again~\eqref{eq:v_one_dir} holds (with $P = a \otimes \xi$ and for $x \in C$), whereas if $\rk P \geq 2$, then $v$ must even be affine. This rigidity result traces its origins to Hadamard's jump condition and Proposition~2 in~\cite{BalJam87FPMM}. For our purposes, however, we need a stronger statement than in the latter reference, but the proof is still elementary and based only on the fact that $\BV$-derivatives must be curl-free, which translates into an algebraic condition on $P$, and the fact that gradients are always othogonal to level sets. To the best of the author's knowledge, rigidity results seem not to have been employed explicitly in lower semicontinuity theory before (except the aforementioned use of Alberti's Theorem of course).

For a sequence $u_j \toweakstar u$ in the $\BV$-Lower Semicontinuity Theorem, we distinguish several different types of blow-up, depending on whether $x_0$ is a regular point for $Du$, or a singular point (see Section~\ref{sc:setup} for definitions), and in the latter case also depending on whether $\rk P(x_0) \leq 1$ or $\rk P(x_0) \geq 2$ for the matrix $P(x_0) := \frac{D^s u}{\abs{D^s u}}(x_0)$. At regular points ($\Lcal^d$-a.e.), we have a \enquote{regular blow-up}, that is an affine blow-up limit, and we can apply quasiconvexity directly. At singular points $x_0 \in \Omega$ with $\rk P(x_0) \leq 1$ we have a \enquote{fully singular blow-up}, meaning that we get a one-directional function in the blow-up limit and we need an averaging procedure before we can apply quasiconvexity (just as in the usual proof). If $\rk P(x_0) \geq 2$, we call this a \enquote{semi-regular blow-up}, we get an affine function in the blow-up again, which can then be treated by just slightly adapting the procedure for regular blow-ups. Of course, from Alberti's Theorem we know that this case occurs only on a $\abs{D^s u}$-negligible set, but the main objective of this work is to avoid using this result.

Our proof is set in the theory of generalized Young measures as introduced by DiPerna \& Majda~\cite{DiPMaj87OCWS} and further developed by Alibert \& Bouchitt\'{e}~\cite{AliBou97NUIG} and others~\cite{KruRou97MDM,Roub97ROTV,KriRin10CGGY}. We follow the framework as presented in~\cite{KriRin10CGGY}. The main reason for choosing this Young measure approach is that it provides a very conceptual and clean organization of the lower semicontinuity proof (and only through this point of view it became apparent to the author how to argue without Alberti's Theorem). In fact, we prove Jensen-type inequalities for the regular and the singular part of a generalized Young measure generated by a sequence of $\BV$-derivatives and then deduce lower semicontinuity from that.

It should be remarked that for the present result it is possible to circumvent the use of Young measures altogether by simply substituting our Rigidity Lemma~\ref{lem:BV_rigidity} in place of Alberti's Theorem in the classical proof (see Remark~\ref{rem:YM_free_lsc} for more details). However, while the Young measure approach requires a few technical results, it obviates the need to use certain other measure-theoretic arguments (like the De Giorgi--Letta Theorem). Besides, a secondary aim of this work is to showcase this Young measure approach, since it is also useful for proving \emph{new} lower semicontinuity results; for instance, the recent proof of lower semicontinuity for symmetric-quasiconvex integral functionals with linear growth in the space BD of functions of bounded deformation~\cite{Rind10?LSIF} relies substantially on the theory of generalized Young measures.

As a noteworthy technical tool we introduce tangent Young measures, which complement classical tangent measures in blow-up arguments involving Young measures. They retain the good compactness properties of weak*-convergence, but contain much more information about the blow-up sequence. Tangent Young measures allow us to formulate Localization Principles for Young measures at regular and singular points. Moreover, we provide a slightly stronger version (and a different proof) of a lemma on strictly converging blow-up sequences, first noticed by Larsen, see Lemma~5.1 of~\cite{Lars98QWOJ}, which allows to shorten the blow-up argument.

The paper is organized as follows: We collect preliminaries and notation in Section~\ref{sc:setup}. Section~\ref{sc:rigidity} recalls basic facts on tangent measures, proves the Strict Blow-up Lemma and then exhibits global and local versions of the key rigidity result. The Localization Principles for Young measures are the topic of Section~\ref{sc:localization}, and, finally, Section~\ref{sc:Jensen} shows the Jensen-type inequalities and deduces the $\BV$-Lower Semicontinuity Theorem~\ref{thm:BV_lsc_teaser} and related results from them.

\section{Setup}  \label{sc:setup}

\subsection{Notation} \label{ssc:general}

By $\Bbb^d$ we denote the open unit ball in $\R^d$ and also set $B(x_0,r) := x_0 + r\Bbb^d$; $\Sbb^{d-1} := \partial \Bbb^d$ is the unit sphere. Generically, $\Omega$ is an open, bounded set in $\R^d$ with Lipschitz boundary. We equip the space $\R^{m \times n}$ of $(m \times n)$-dimensional matrices with the Frobenius norm $\abs{A} := \sqrt{\sum_{i,j} (A_j^i)^2} = \sqrt{\mathrm{trace}(A^T A)}$ (the Euclidean norm in $\R^{mn}$), where $A_j^i$ denotes the entry of $A$ in the $i$th row and $j$th column. The \term{tensor product} between two vectors $a \in \R^m$, $b \in \R^d$ is $a \otimes b := ab^T \in \R^{m \times d}$. For a set $A$, $\ONE_A$ is its \term{characteristic function}, while we use $\ONE$ to denote the function, which is constant and equal to $1$ everywhere.

The space $\Mbf(\R^d;\R^{m \times d})$ contains all \term{finite (Radon) measures} on the Borel $\sigma$-algebra of $\R^d$ with values in $\R^{m \times d}$. Analogously, define $\Mbf(A;\R^{m \times d})$ for a Borel set $A \subset \R^d$. For every measure $\mu \in \Mbf(\R^d;\R^{m \times d})$, we denote by $\abs{\mu} \in \Mbf(\R^d)$ its \term{total variation measure}. We write $\dpr{h,\mu} := \int h \cdot \di \mu$ for any Borel measurable $h \colon \R^d \to \R^{m \times d}$. The symbols $\Lcal^d$ and $\Hcal^k$ denote the $d$-dimensional Lebesgue measure and the $k$-dimensional Hausdorff measure ($0 \leq k < \infty$), respectively. The restriction $\mu \restrict B$ for a measure $\Mbf(A;\R^{m \times d})$ and a Borel set $B \subset A$ is defined via $(\mu \restrict B)(C) := \mu(C \cap B)$ for every Borel set $C \subset A$.

Every measure $\mu \in \Mbf(\R^d;\R^{m \times d})$ has a \term{Lebesgue--Radon--Nikod\'{y}m decomposition}
\[
  \mu = \mu^a + \mu^s = \frac{\di \mu}{\di \Lcal^d} \Lcal^d + \frac{\di \mu^s}{\di \abs{\mu^s}} \abs{\mu^s},
\]
the function $\frac{\di \mu^s}{\di \abs{\mu^s}} \in \Lrm^1(\R^d,\abs{\mu^s};\partial \Bbb^{m \times d})$ is also referred to as the \term{polar function} of $\mu^s$. Here, the space $\Lrm^1(\R^d,\abs{\mu^s};\partial \Bbb^{m \times d})$ contains all $\abs{\mu^s}$-integrable functions with values in the unit sphere of $\R^{m \times d}$. More on these notions can for example be found in~\cite{AmFuPa00FBVF,FonLeo07MMCV}.

We call $x_0 \in \supp \mu$ a \term{regular point} if the \term{Radon--Nikod\'{y}m derivative} $\frac{\di \mu}{\di \Lcal^d}(x_0)$ exists as the limit
\[
  \frac{\di \mu}{\di \Lcal^d}(x_0) =
  \lim_{r \todown 0} \frac{\mu(B(x_0,r))}{\abs{B(x_0,r)}} \quad \in \R^{m \times d},
\]
where $\abs{B(x_0,r)} = \Lcal^d(B(x_0,r))$ is the $d$-dimensional Lebesgue measure of the ball $B(x_0,r)$. It is well-known that $\Lcal^d$-almost all $x_0 \in \supp \mu$ are regular points, the other points are called \term{singular points}.

On several occasions we will employ the \term{pushforward} measure $T_*^{(x_0,r)} \mu := \mu \circ (T^{(x_0,r)})^{-1} \in \Mbf(\R^d;\R^{m \times d})$ of a finite Radon measure $\mu \in \Mbf(\R^d;\R^{m \times d})$ under the affine mapping $T^{(x_0,r)}(x) := (x-x_0)/r$, where $x_0 \in \R^d$ and $r > 0$. For a (measurable) function $g \colon \R^d \to \R$ we have the transformation formula
\[
  \int g \dd (T_*^{(x_0,r)}\mu) = \int g \circ T^{(x_0,r)} \dd \mu
  = \int g \Bigl(\frac{x-x_0}{r}\Bigr) \dd \mu(x),
\]
provided one, hence all, of these integrals are defined.

Besides the usual \term{weak* convergence} $\mu_j \toweakstar \mu$ of a sequence of measures $(\mu_j) \subset \Mbf(\R^d;\R^{m \times d})$, we also use the \term{strict convergence} where in addition to weak* convergence $\mu_j \toweakstar \mu$ we also assume $\abs{\mu_j}(\R^d) \to \abs{\mu}(\R^d)$ (or $\abs{\mu_j}(A) \to \abs{\mu}(A)$ if we consider measures on a Borel set $A$).

By $\BV(\Omega;\R^m)$ we denote the space of \term{functions of bounded variation}, i.e.\ the space of functions $u \in \Lrm^1(\Omega;\R^m)$ such that the distributional derivative $Du$ is (representable as) a finite matrix-valued Radon measure, $Du \in \Mbf(\Omega;\R^{m \times d})$. We write its Lebesgue--Radon--Nikod\'{y}m decomposition as $Du = \nabla u \Lcal^d + D^s u$ and call $\nabla u$ the \term{approximate gradient} (more precisely, the Lebesgue-density coincides a.e.\ with the approximate gradient, which is defined in a pointwise fashion), while $D^s u$ is the \term{singular part} of the derivative. We use the weak* and the strict convergence in $\BV(\Omega;\R^m)$, which correspond to $\Lrm^1$-convergence together with respectively weak* or strict convergence of the derivatives. Finally, each function $u \in \BV(\Omega;\R^m)$, $\Omega \subset \R^d$ an open bounded Lipschitz domain as usual, has a boundary trace $u|_{\partial \Omega} \in \Lrm^1(\partial \Omega, \Hcal^{d-1} \restrict \partial \Omega;\R^m)$ and the trace operator $u \mapsto u|_{\partial \Omega}$ is strictly, but not weakly* continuous. A thorough introduction to the space $\BV(\Omega;\R^m)$ is given in~\cite{AmFuPa00FBVF}.

\subsection{Integrands}

For $f \colon \Omega \times \R^{m \times d} \to \R$ with \term{linear growth at infinity}, that is $\abs{f(x,A)} \leq M(1+\abs{A})$ for some constant $M \geq 0$ and all $x \in \Omega$, $A \in \R^{m \times d}$, define the transformation
\[
  (Sf)(x,\hat{A}) := (1-\abs{\hat{A}})f \Bigl( x, \frac{\hat{A}}{1-\abs{\hat{A}}}
  \Bigr),  \qquad\text{$x \in \Omega$, $\hat{A} \in \Bbb^{m \times d}$,}
\]
where $\Bbb^{m \times d}$ is the unit ball in $\Rbb^{m \times d}$. Then $Sf \colon \Omega \times \Bbb^{m \times d} \to \R$, and we let $\Ebf(\Omega;\R^{m \times d})$ be the space of all $f \in \Crm(\Omega \times \R^{m \times d})$ such that $Sf$ extends into a bounded, continuous function on $\cl{\Omega \times \Bbb^{m \times d}}$ (which is equivalent to $Sf$ being uniformly continuous on $\Omega \times \Bbb^{m \times d}$). We norm this space by
\[
  \norm{f}_{\Ebf(\Omega;\R^{m \times d})} := \sup_{(x,\hat{A}) \in \Omega \times \Bbb^{m \times d}} \absb{Sf(x,\hat{A})},
  \qquad f \in \Ebf(\Omega;\R^{m \times d}).
\]
From the definition we get that for each $f \in \Ebf(\Omega;\R^{m \times d})$ the limit
\begin{equation} \label{eq:f_infty}
  f^\infty(x,A) := \lim_{\substack{\!\!\!\! x' \to x \\ \!\!\!\! A' \to A \\ \; t \to \infty}}
    \frac{f(x',tA')}{t},  \qquad \text{$x \in \cl{\Omega}$, $A \in \R^{m \times d}$,}
\end{equation}
exists and defines a positively $1$-homogeneous function (i.e.\ $f(x,\theta A) = \theta f(x,A)$ for all $\theta \geq 0$, $x \in \cl{\Omega}$, $A \in \R^{m \times d}$). This function $f^\infty$ is called the \term{(strong) recession function} of $f$.

If we only have $h \in \Crm(\R^{m \times d})$ with linear growth at infinity, the recession function $h^\infty$ does not necessarily exist (not even for \emph{quasiconvex} $h \in \Crm(\R^{m \times d})$, see Theorem~2 of~\cite{Mull92QFHD}). But for such functions $h$ we can always define the \term{generalized recession function} $h^\# \colon \R^{m \times d} \to \R$ by
\[
  h^\#(A) := \limsup_{\substack{\!\!\!\! A' \to A \\ \; t \to \infty}} \frac{h(tA')}{t},  \qquad A \in \R^{m \times d},
\]
which again is always positively $1$-homogeneous ($h^\#$ is usually just called the \enquote{recession function} in other works, but here the distinction is important). We refer to Section~2.5 of~\cite{AusTeb03ACFO} for a more systematic approach to recession functions and their associated cones.

If $f,h$ are Lipschitz continuous, then the definitions of $f^\infty$ and $h^\#$ simplify to
\begin{equation}  \label{eq:rec_funct_simpler}
\begin{aligned}
  f^\infty(x,A) &= \lim_{\substack{\!\!\!\! x' \to x \\ \; t \to \infty}}
   \frac{f(x',tA)}{t}, \qquad \text{$x \in \cl{\Omega}$, $A \in \R^{m \times d}$,} \\
  h^\#(A) &= \limsup_{t \to \infty} \frac{h(tA)}{t},  \qquad A \in \R^{m \times d}.
\end{aligned}
\end{equation}
It is elementary to show that $h^\#$ is always upper semicontinuous.

We will also need the following approximation lemma:

\begin{lemma} \label{lem:integrand_approx}
For every upper semicontinuous function $h \colon \R^{m \times d} \to \R$ with linear growth at infinity, there exists a decreasing sequence $(\ONE \otimes h_k) \subset \Ebf(\Omega;\R^{m \times d})$ with
\[
  \inf_{k \in \N} h_k = \lim_{k \to \infty} h_k = h, \qquad
  \inf_{k \in \N} h_k^\infty = \lim_{k \to \infty} h_k^\infty = h^\#
  \qquad\text{(pointwise).}
\]
Furthermore, the linear growth constants of the $h_k$ can be chosen to be bounded by the linear growth constant of $h$.
\end{lemma}

A proof can be found in Lemma 2.3 of~\cite{AliBou97NUIG} or the appendix of~\cite{KriRin10RSIF}.

\subsection{Young measures} \label{ssc:YM}

Generalized Young measures were introduced by DiPerna \& Majda in~\cite{DiPMaj87OCWS}, we here follow the framework of~\cite{KriRin10CGGY}, which itself is based upon Alibert \& Bouchitt\'{e}'s reformulation~\cite{AliBou97NUIG} of the theory.

A \term{(generalized) Young measure} carried by the open set $\Omega \subset \R^d$ and with values in $\R^{m \times d}$ is a triple $(\nu_x,\lambda_\nu,\nu_x^\infty)$, where
\begin{itemize}
  \item[(i)] $(\nu_x)_{x \in \Omega} \subset \Mbf(\R^{m \times d})$ is a parametrized family of probability measures on $\R^{m \times d}$,
  \item[(ii)] $\lambda_\nu \in \Mbf(\cl{\Omega})$ is a positive finite measure on $\cl{\Omega}$, and
  \item[(iii)] $(\nu_x^\infty)_{x \in \cl{\Omega}} \subset \Mbf(\partial \Bbb^{m \times d})$ is a parametrized family of probability measures on the unit sphere $\partial \Bbb^{m \times d}$ of $\R^{m \times d}$.
\end{itemize}
Moreover, we require that
\begin{itemize}
  \item[(iv)] the map $x \mapsto \nu_x$ is weakly* measurable with respect to $\Lcal^d$, i.e.\ the function $x \mapsto \dpr{f(x,\frarg),\nu_x}$ is Lebesgue-measurable for every bounded Borel function $f \colon \Omega \times \R^{m \times d} \to \R$,
  \item[(v)] the map $x \mapsto \nu_x^\infty$ is weakly* measurable with respect to $\lambda_\nu$, and
  \item[(vi)] $x \mapsto \dprn{\abs{\frarg},\nu_x} \in \Lrm^1(\Omega)$.
\end{itemize}
We collect all such Young measures in the set $\Ybf(\Omega;\R^{m \times d})$.

For an integrand $f \in \Ebf(\Omega;\R^{m \times d})$ and a Young measure $\nu \in \Ybf(\Omega;\R^{m \times d})$ we set
\[
  \ddprb{f,\nu} := \int_\Omega \dprb{f(x,\frarg),\nu_x} \dd x
  + \int_{\cl{\Omega}} \dprb{f^\infty(x,\frarg),\nu_x^\infty} \dd \lambda_\nu(x).
\]
Since $\Ybf(\Omega;\R^{m \times d})$ is part of the dual space to $\Ebf(\Omega;\R^{m \times d})$ via the duality pairing $\ddpr{\frarg,\frarg}$, we say that a sequence of Young measures $(\nu_j) \subset \Ybf(\Omega;\R^{m \times d})$ \term{converges weakly*} to $\nu \in \Ybf(\Omega;\R^{m \times d})$, in symbols $\nu_j \toweakstar \nu$, if
\begin{equation} \label{eq:nuj_conv}
  \ddprb{f,\nu_j} \to \ddprb{f,\nu}  \qquad\text{for all $f \in \Ebf(\Omega;\R^{m \times d})$.}
\end{equation}

Fundamental for all Young measure theory is the following Compactness Theorem, see Section~3.1 of~\cite{KriRin10CGGY} for a proof:

\begin{theorem} \label{thm:YM_compactness}
Let $(\nu_j) \subset \Ybf(\Omega;\R^{m \times d})$ be a sequence of Young measures satisfying
\begin{itemize}
  \item[(i)] the functions $x \mapsto \dprn{\abs{\frarg},(\nu_j)_x}$ are uniformly bounded in $\Lrm^1(\Omega)$,
  \item[(ii)] $\sup_j \lambda_{\nu_j}(\cl{\Omega}) < \infty$,
\end{itemize}
or, equivalently,
\[
  \textstyle\sup_j \, \ddprb{\ONE \otimes \abs{\frarg}, \nu_j} < \infty.
\]
Then, there exists a subsequence (not relabeled) and $\nu \in \Ybf(\Omega;\R^{m \times d})$ such that $\nu_j \toweakstar \nu$ in $\Ybf(\Omega;\R^{m \times d})$.
\end{theorem}

The following density (or separability) lemma is proved in Lemma~3 of~\cite{KriRin10CGGY}:

\begin{lemma} \label{lem:E_separable}
There exists a countable set of functions $\{f_k\} = \setn{ \phi_k \otimes h_k \in \Crm(\cl{\Omega}) \times \Crm(\R^{m \times d}) }{ k \in \N } \subset \Ebf(\Omega;\R^{m \times d})$ such that $\ddpr{f_k,\nu_1} = \ddpr{f_k,\nu_2}$ for two Young measures $\nu_1,\nu_2 \in \Ybf(\Omega;\R^{m \times d})$ and all $k \in \N$, implies $\nu_1 = \nu_2$. Moreover, all the $h_k$ can be chosen Lipschitz continuous.
\end{lemma}

An immediate consequence is that to determine the limit in the weak* convergence $\nu_j \toweakstar \nu$ in $\Ebf(\Omega;\R^{m \times d})$ of a bounded Young measure sequence in the sense of conditions (i), (ii) in Theorem~\ref{thm:YM_compactness}, it suffices to test with the collection $\{f_k\}$ exhibited in the previous lemma.

Each measure $\mu \in \Mbf(\cl{\Omega};\R^{m \times d})$ with Lebesgue--Radon--Nikod\'{y}m decomposition $\mu = a \Lcal^d \restrict \Omega + p \abs{\mu^s}$, where $a \in \Lrm^1(\Omega;\R^{m \times d})$, $p \in \Lrm^1(\cl{\Omega},\abs{\mu^s};\partial \Bbb^{m \times d})$, induces an \term{elementary Young measure} $\epsilon_\mu \in \Ybf(\Omega;\R^{m \times d})$ through
\[
  (\epsilon_\mu)_x := \delta_{a(x)},  \qquad \lambda_{\epsilon_\mu} := \abs{\mu^s},
    \qquad (\epsilon_\mu)_x^\infty := \delta_{p(x)}.
\]
If $\epsilon_{\mu_j} \toweakstar \nu$ in $\Ybf(\Omega;\R^{m \times d})$, then we say that the $\mu_j$ \term{generate} $\nu$ and we write $\mu_j \toY \nu$.

The limit representation~\eqref{eq:nuj_conv} can be extended as follows, see Proposition~2 of~\cite{KriRin10CGGY} for a proof.

\begin{proposition} \label{prop:ext_repr}
Let $\nu_j \toweakstar \nu$ in $\Ybf(\Omega;\R^{m \times d})$. Then, $\ddpr{f,\nu_j} \to \ddpr{f,\nu}$ 
holds provided one of the following conditions is satisfied:
\begin{enumerate}
  \item[(i)] $f \colon \cl{\Omega} \times \R^{m \times d} \to \R$ is a Carath\'{e}odory integrand possessing a recession function $f^\infty$ in the sense of~\eqref{eq:f_infty} that is (jointly) continuous.
  \item[(ii)] $f(x,A) = \ONE_U(x) g(x,A)$, where $g \in \Ebf(\Omega;\R^n)$ and a Borel set $U \subset \Omega$ with $(\Lcal^d + \lambda_\nu)(\partial U) = 0$,
\end{enumerate}
\end{proposition}

The \term{barycenter} $[\nu] \in \Mbf(\R^d;\R^{m \times d})$ of a Young measure $\nu \in \Ybf(\R^d;\R^{m \times d})$, is
\[
  [\nu] := \dprb{\id,\nu_x} \, \Lcal^d + \dprb{\id,\nu_x^\infty} \, \lambda_\nu.
\]
Clearly, if $\mu_j \toweakstar \mu$ in $\Mbf(\cl{\Omega};\R^{m \times d})$ and $\mu_j \toY \nu \in \Ybf(\Omega;\R^{m \times d})$, then $[\nu] = \mu$.

\subsection{Gradient Young measures} \label{ssc:GYM}

In this paper, we are only interested in Young measures that are generated by a sequence of $\Wrm^{1,1}$-gradients or $\BV$-derivatives. We define the set $\GYbf(\Omega;\R^{m \times d})$ of \term{gradient Young measures} to be the set of all Young measures $\nu \in \Ybf(\Omega;\R^{m \times d})$ such that there exists a (necessarily norm-bounded) sequence $(u_j) \subset \BV(\Omega;\R^m)$ with $Du_j \toY \nu$.

We have the following theorem on generation, which is an immediate consequence of the above Compactness Theorem:

\begin{theorem}
Let $(u_j) \subset \BV(\Omega;\R^m)$ be a uniformly norm-bounded sequence, that is $\sup_j (\norm{u_j}_{\Lrm^1(\Omega;\R^m)} + \abs{Du_j}(\Omega)) < \infty$. Then, there exists a subsequence of the $u_j$ (not relabeled) such that $Du_j \toY \nu \in \GYbf(\Omega;\R^{m \times d})$.
\end{theorem}

By mollification (see Proposition~4 of~\cite{KriRin10CGGY}) it is proved that for every $\nu \in \GYbf(\Omega;\R^{m \times d})$, there also exists a sequence $(v_j) \subset (\Wrm^{1,1} \cap \Crm^\infty)(\Omega;\R^m)$ with $\nabla v_j \toY \nu$ (which is of course to be understood as $\nabla v_j \Lcal^d \restrict \Omega \toY \nu$). In fact, even the following stronger statement is true:

\begin{lemma} \label{lem:boundary_adjust}
Let $\nu \in \GYbf(\Omega;\R^{m \times d})$ be a gradient Young measure with $\lambda_\nu(\partial \Omega) = 0$ and barycenter $[\nu] = Du$, where $u \in \BV(\Omega;\R^m)$. Then, there exists a generating sequence $(v_j) \subset (\Wrm^{1,1} \cap \Crm^{\infty})(\Omega;\R^m)$ with $Dv_j \toY \nu$, $v_j \toweakstar u$ in $\BV(\Omega;\R^m)$, and $v_j|_{\partial \Omega} = u|_{\partial \Omega}$ (in the sense of trace) for all $j \in \N$.
\end{lemma}

This boundary adjustment is standard, a detailed proof can be found in Lemma~4 of~\cite{KriRin10CGGY}.

\subsection{Quasiconvexity}  \label{ssc:qc}

A locally bounded Borel function $h \colon \R^{m \times d} \to \R$ is called \term{quasiconvex} if
\[
  h(A) \leq \dashint_{\omega} h \bigl( A + \nabla \psi(x) \bigr) \dd x
  \qquad\text{for all $A \in \R^{m \times d}$ and all
  $\psi \in \Crm_0^\infty(\omega;\R^m)$,}
\]
where $\omega \subset \R^d$ is an arbitrary bounded Lipschitz domain, and $\Crm_0^\infty(\omega;\R^m)$ is the set of infinitely differentiable functions with zero boundary values. By standard covering arguments it suffices to check this for one particular choice of $\omega$ only. Moreover, if $h$ has linear growth at infinity, the requirement that $\psi \in \Crm_0^\infty(\omega;\R^m)$ may equivalently be replaced by $\psi \in \Wrm_0^{1,1}(\omega;\R^m)$. See~\cite{Daco08DMCV} for details on quasiconvexity.

It is well-known that quasiconvex functions are rank-one convex, i.e.\ convex along rank-one lines. Notice also that under the assumption of linear growth it follows from Fatou's Lemma that the generalized recession function $h^\#$ is quasiconvex whenever $h$ is, see for example~\cite{AmFuPa00FBVF},~pp.~303--304; the same applies to $f^\infty$ if it exists (quasiconvexity then is understood with respect to the second argument).

It is shown in Morrey's book~\cite{Morr66MICV}, see also Lemma~2.2 of~\cite{BaKiKr00RQE}, that quasiconvex functions with linear growth are Lipschitz continuous, hence we may use the simpler definition~\eqref{eq:rec_funct_simpler} for the recession functions.

\section{Tangent measures, rigidity, and fine structure of $\BV$-derivatives}  \label{sc:rigidity}

\subsection{Tangent measures}

Let $T^{(x_0,r)}(x) := (x-x_0)/r$ for $x_0 \in \R^d$ and $r > 0$. For a matrix-valued measure $\mu \in \Mbf(\R^d;\R^{m \times d})$ and $x_0 \in \supp \mu$, we call a \term{tangent measure} to $\mu$ in $x_0$ any weak* limit of (restrictions of) the rescaled measures $c_n T^{(x_0,r_n)}_* \mu$ in $\Mbf(\Bbb^d;\R^{m \times d})$, where $r_n \todown 0$ is a sequence of radii and $c_n := \abs{\mu}(B(x_0,r_n))^{-1}$. The set of all such tangent measures is denoted by $\Tan(\mu,x_0)$ and the sequence $c_n T^{(x_0,r_n)}_* \mu$ is called a \term{blow-up sequence}. General information on the above definition of tangent measures can for example be found in Chapter~2 of~\cite{AmFuPa00FBVF}, whereas in~\cite{Matt95GSME} one can find much information on Preiss's original, more general definition of tangent measures and their applications in Geometric Measure Theory (see~\cite{Prei87GMDR} for the original work).

At $\abs{\mu}$-almost every point $x_0 \in \supp \mu$, there exists a sequence $r_n \todown 0$ such that the condition
\begin{equation} \label{eq:blowup_loc_bdd}
  \limsup_{n \to \infty} \frac{\abs{\mu}(B(x_0,Kr_n))}{\abs{\mu}(B(x_0,r_n))}
  \leq \beta_K
\end{equation}
is satisfied for all $K \in \N$ and some constants $\beta_K \geq 0$; this is proved in Lemma~2.4, Theorem~2.5 of~\cite{Prei87GMDR} (or see the appendix to~\cite{Rind10?LSIF}).

A special property of tangent measures is that at $\abs{\mu}$-almost every $x_0 \in \R^d$ and any $r_n \todown 0$ it holds that
\begin{equation} \label{eq:Tan_abs}
  \tau = \wslim_{n\to\infty} \, c_n T_*^{(x_0,r_n)} \mu  \quad\text{if and only if}\quad
  \abs{\tau} = \wslim_{n\to\infty} \, c_n T_*^{(x_0,r_n)} \abs{\mu},
\end{equation}
where the weak* limits are to be understood in the spaces $\Mbf(\Bbb^d;\R^{m \times d})$ and $\Mbf(\Bbb^d;\R)$, respectively, see for instance Theorem~2.44 in~\cite{AmFuPa00FBVF} for a proof. The previous equivalence in particular entails
\begin{equation} \label{eq:Tan_density}
  \Tan(\mu,x_0) = \frac{\di \mu}{\di \abs{\mu}}(x_0) \cdot \Tan(\abs{\mu},x_0).
\end{equation}

We will also need to employ tangent measures which are not defined on the unit ball $\Bbb^d$, but on some other open convex set $C \subset \R^d$ containing the origin. In this case, we write $\Tan_C(\mu,x_0)$ instead of $\Tan(\mu,x_0)$ and set $c_n :=  \abs{\mu}(C(x_0,r_n))^{-1}$, where $C(x_0,r) := x_0 + rC$. Clearly, analogous statements to before hold.

At a given point $x_0 \in \supp \mu$, different blow-up sequences might behave very differently. Most starkly, this phenomenon can be observed for the (positive) O'Neil measure~\cite{ONei95MLST}, which has \emph{every} non-zero sub-probability measure as tangent measure at almost every point. Therefore, we need to distinguish several classes of blow-up sequences $\gamma_j := c_j T^{(x_0,r_j)}_* \mu$ for a measure $\mu \in \Mbf(\R^d;\R^{m \times d})$ at a point $x_0 \in \supp \mu$:
\begin{itemize}
  \item \term{Regular blow-up:} $x_0$ is a regular point of $\mu$ and $\gamma_j \toweakstar A_0 \Lcal^d \restrict C$, where $A_0 = \frac{\di \mu}{\di \Lcal^d}(x_0)$.
  \item \term{Semi-regular blow-up:} $x_0$ is a singular point of $\mu$, but nevertheless $\gamma_j \toweakstar A_0 \Lcal^d  \restrict C$, where $A_0 = \frac{\di \mu}{\di \abs{\mu}}(x_0)$.
  \item \term{Fully singular blow-up:} $x_0$ is a singular point of $\mu$ and $\gamma_j \toweakstar \tau$, but $\tau \neq A \Lcal^d \restrict C$ for any $A \in \R^{m \times d}$.
\end{itemize}
At $\Lcal^d$-almost every Lebesgue point $x_0 \in \supp \mu$ of $\frac{\di \mu}{\di \Lcal^d}$ with respect to $\Lcal^d$ we have that $\Tan_C(\mu,x_0)$ contains only one measure, which is a constant multiple of $\Lcal^d \restrict C$. Hence, at $\Lcal^d$-almost every $x_0 \in \supp \mu$, all blow-up sequences are regular.

In some sense conversely to the O'Neil measure alluded to above, Preiss exhibited a positive, purely singular measure on a bounded interval (in particular a $\BV$-derivative) such that all tangent measures are a fixed multiple of Lebesgue measure, see Example 5.9(1) in~\cite{Prei87GMDR}. In our terminology above this means that all blow-ups at almost all the singular points are semi-regular.

The following result seems to have appeared first in Lemma~5.1 of~\cite{Lars98QWOJ}; we here give a slightly stronger version with a different proof.

\begin{lemma}[Strictly converging blow-ups] \label{lem:strict_blowup}
Let $\mu \in \Mbf(\R^d;\R^{m \times d})$. For $\abs{\mu}$-almost every $x_0 \in \supp \mu$ the following two assertions holds for all open convex sets $C \subset \R^d$:
\begin{itemize}
  \item[(i)] There exists $\tau \in \Tan_C(\mu,x_0)$ with $\abs{\tau}(C) = 1$, $\abs{\tau}(\partial C) = 0$.
  \item[(ii)] Moreover, there exists a sequence $r_n \todown 0$, such that the blow-up sequence $\gamma_n := c_n T_*^{(x_0,r_n)} \mu$ (as usual, $c_n := \abs{\mu}(C(x_0,r_n))^{-1}$) satisfies $\abs{\gamma_n}(\partial C) = 0$ and $\gamma_n \to \tau$ strictly in $\Mbf(\cl{C};\R^{m \times d})$
\end{itemize}
\end{lemma}

For the assertions $\abs{\tau}(\partial C) = 0$, $\abs{\gamma_n}(\partial C) = 0$, $\tau$ and $\gamma_n$ are to be considered as measures on $\cl{C}$.

\begin{proof}
Let $x_0 \in \supp \mu$ be such that condition~\eqref{eq:blowup_loc_bdd} is satisfied for a sequence $\rho_n \todown 0$; this is the case for $\abs{\mu}$-almost every $x_0 \in \Omega$.

Pick $\eta > 0$, $K \in \N$ such that $\Bbb^d \subset\subset \eta C \subset\subset K \Bbb^d$. Then, let $a_n T_*^{(x_0,\rho_n)} \mu \toweakstar \tilde{\tau}$ in $\Mbf(K\Bbb^d;\R^{m \times d})$, possibly selecting a subsequence of the $\rho_n$s, and with $a_n := \abs{\mu}(B(x_0,\rho_n))^{-1}$. Hence, $\abs{\tilde{\tau}}(\eta C) > 0$. By~\eqref{eq:Tan_abs}, also $a_n T_*^{(x_0,\rho_n)} \abs{\mu} \toweakstar \abs{\tilde{\tau}}$ and slightly increasing $\eta$ if necessary, we may also assume $\abs{\tilde{\tau}}(\partial (\eta C)) = 0$.

Set $b_n := \abs{\tilde{\tau}}(\eta C)^{-1} a_n$, $\tau := \abs{\tilde{\tau}}(\eta C)^{-1} \tilde{\tau}(\eta \frarg) = \abs{\tilde{\tau}}(\eta C)^{-1} T_*^{(0,\eta)} \tilde{\tau}$, and observe
\[
  b_n T_*^{(x_0,\eta\rho_n)} \mu \toweakstar \tau,  \qquad
  b_n T_*^{(x_0,\eta \rho_n)} \abs{\mu} \toweakstar \abs{\tau}
  \qquad\text{in $\Mbf(\cl{C};\R^{m \times d})$}.
\]
From $\abs{\tau}(\partial C) = \abs{\tilde{\tau}}(\eta C)^{-1} \abs{\tilde{\tau}}(\partial (\eta C)) = 0$, standard results in measure theory allow us to infer
\begin{equation} \label{eq:eta_rho_j_strict}
  b_n T_*^{(x_0,\eta \rho_n)} \abs{\mu}(C) \to \abs{\tau}(C) = 1.
\end{equation}

Now pick a sequence $r_n \todown 0$ such that
\begin{itemize}
\item[(I)] $\eta \rho_n \leq r_n \leq \eta \rho_n(1+1/(b_n n))$,
\item[(II)] $T_*^{(x_0,r_n)} \mu(\partial C) = 0$ (from finiteness), and
\item[(III)] $b_n T_*^{(x_0,r_n)} \absb{\mu}(C) \leq b_n T_*^{(x_0,\eta \rho_n)} \absb{\mu}(C) + 1/n$ (from outer regularity).
\end{itemize}
Then, for all $\psi \in \Crm_0^1(K\Bbb^d;\R^{m \times d})$ with Lipschitz constant $L$ say,
\begin{align*}
  &b_n \absBB{ \! \int \psi \cdot \di T_*^{(x_0,r_n)} \mu
    - \int \! \psi \cdot \di T_*^{(x_0,\eta \rho_n)} \mu} \\
  &\qquad \leq b_n \int \absBB{ \psi \Bigl(\frac{x-x_0}{r_n}\Bigr)
    - \psi \Bigl(\frac{x-x_0}{\eta \rho_n}\Bigr) } \dd \abs{\mu}(x) \\
  &\qquad \leq b_n L \int_{B(x_0,K r_n)} \absBB{\frac{x-x_0}{r_n} - \frac{x-x_0}{\eta \rho_n}} \dd \abs{\mu}(x),
\end{align*}
and this goes to zero as $n \to \infty$, because by (I)
\[
  \absBB{\frac{x-x_0}{r_n} - \frac{x-x_0}{\eta \rho_n}}
  \leq \frac{\abs{\eta \rho_n - r_n}}{\eta \rho_n} \cdot
  \absBB{\frac{x-x_0}{r_n}} \leq \frac{K}{b_n n}
  \qquad\text{for all $x \in B(x_0,K r_n)$}
\]
Thus, in particular,
\begin{equation} \label{eq:bn_blowup}
\begin{aligned}
  &b_n T_*^{(x_0,r_n)} \mu \toweakstar \tau
  \quad\text{in $\Mbf(\cl{C};\R^{m \times d})$} \qquad \text{and} \\
  &b_n T_*^{(x_0,r_n)} \abs{\mu}(C) \to \abs{\tau}(C) = 1,
\end{aligned}
\end{equation}
where for the second assertion we also used (III) together with~\eqref{eq:eta_rho_j_strict}.

It remains to show that $c_n T_*^{(x_0,r_n)} \mu \toweakstar \tau$ in $\Mbf(C;\R^{m \times d})$. To this effect observe
\[
  1 = \lim_{n\to\infty} c_n T_*^{(x_0,r_n)} \abs{\mu}(C)
  = \lim_{n\to\infty} \frac{c_n}{b_n} \cdot \lim_{n\to\infty}
  b_n T_*^{(x_0,r_n)} \abs{\mu}(C)
  = \lim_{n\to\infty} \frac{c_n}{b_n}
\]
and hence we may replace $b_n$ by $c_n$ in~\eqref{eq:bn_blowup}.
\end{proof}

\subsection{Rigidity}

In this section we establish that functions $u \in \BV(\Omega;\R^m)$ with the property that $Du = P \abs{Du}$, where $P \in \R^{m \times d}$ is a fixed matrix, have a very special structure. The origins of this observation can be traced back to Hadamard's jump condition, Proposition~2 in~\cite{BalJam87FPMM}, and Lemma~1.4 of~\cite{DeLe09NARO}; also see the proof of Theorem~3.95 in~\cite{AmFuPa00FBVF}. Other rigidity results may be found in~\cite{Mull99VMMP,Kirc03RGM,KiMuSv03SNPG} and the references cited therein.

\begin{lemma}[Rigidity of $\BV$-functions] \label{lem:BV_rigidity}
Let $C \subset \R^d$ be open and convex (not necessarily bounded), and let $u \in \BV(C;\R^m)$ such that $Du = P \abs{Du}$, or equivalently $\frac{\di Du}{\di \abs{Du}}(x) = P$ almost everywhere, where $P \in \R^{m \times d}$, $\abs{P} = 1$, is a fixed matrix.
\begin{itemize}
  \item[(i)]  If $\rk P \geq 2$, then $u(x) = u_0 + \alpha Px$ (a.e.), where $\alpha \in \R$, $u_0 \in \R^m$.
  \item[(ii)] If $P = a \otimes \xi$ ($a \in \R^m$, $\xi \in \Sbb^{d-1}$), then there exist $\psi \in \BV(\R)$, $u_0 \in \R^m$ such that $u(x) = u_0 + \psi(x \cdot \xi)a$ (a.e.).
\end{itemize}
\end{lemma}

\begin{proof}
First assume $u \in (\Wrm^{1,1} \cap \Crm^\infty)(C;\R^m)$. The idea of the proof is that the curl of $\nabla u$ vanishes, i.e.\
\[
  \partial_i (\nabla u)_j^k = \partial_j (\nabla u)_i^k  \qquad\text{for all $i,j = 1,\ldots,d$ and $k=1,\ldots,m$.}
\]
For our special $\nabla u = P g$, where $g \in \Crm^\infty(C)$ is a smooth function, this gives the conditions
\begin{equation} \label{eq:Curl_ker}
  P_j^k \partial_i g = P_i^k \partial_j g  \qquad\text{for all $i,j = 1,\ldots,d$ and $k=1,\ldots,m$.}
\end{equation}
Under the assumptions of (i), we claim that $\nabla g \equiv 0$. If otherwise $\xi(x) := \nabla g(x) \neq 0$ for some $x \in C$, then with $a_k(x) := P_j^k / \xi_j(x)$ ($k = 1,\ldots,m$) for any $j$ such that $\xi_j(x) \neq 0$ (the quantity $a_k(x)$ is well-defined by the relation~\eqref{eq:Curl_ker}), we have $P_j^k = a_k(x) \xi_j(x)$, which immediately implies $P = a(x) \otimes \xi(x)$. This, however, is impossible if $\rk P \geq 2$. Hence, $\nabla g \equiv 0$ and $u$ is an affine function, which must be of the form exhibited in assertion (i).

For part (ii), that is $P = a \otimes \xi$, we additionally assume $a \neq 0$. Observe that in this case $\nabla g(x) = \theta(x) \xi^T$ for some function $\theta \in \Crm^\infty(C)$. Indeed,~\eqref{eq:Curl_ker} entails
\[
  \xi \nabla g(x) = \nabla g(x)^T \xi^T \in \R^{d \times d},
\]
which gives the projection relation
\[
  \nabla g(x) = (\xi \cdot \nabla g(x)^T)\xi^T =: \theta(x)\xi^T.
\]
Next, since level sets of a function are always orthogonal to the function's gradient, we infer that $g$, hence $\theta$, is constant on all hyperplanes orthogonal to $\xi$ intersected with $C$. As $C$ is assumed convex, we may therefore write
\[
  \theta(x) = \tilde{\theta}(x \cdot \xi),  \qquad x \in C
\]
with $\tilde{\theta} \in \Crm^\infty(\R)$. Taking $\psi \in \Crm^\infty(\R)$ with $\psi'' = \tilde{\theta}$, we have
\[
  \nabla \bigl[ g(x) - \psi'(x \cdot \xi) \bigr]
  = \theta(x)\xi^T - \psi''(x \cdot \xi)\xi^T = 0,
\]
whence $\nabla u(x) = P g(x) = (a \otimes \xi) \psi'(x \cdot \xi)$ (absorb any constant into $\psi'$). Thus, $u(x) = u_0 + \psi(x \cdot \xi)a$ for some $u_0 \in \R^m$.

For general $u$ as in the statement of the proposition, we employ a mollification argument as follows: Consider a convex subdomain $C' \subset\subset C$ and mollify the original $u$ by a smooth kernel with support inside $B(0,d)$, where $d > 0$ is the distance from $C'$ to $\R^d \setminus C$. Then, in $C'$ this yields a smooth function $\tilde{u} \in (\Wrm^{1,1} \cap \Crm^\infty)(C';\R^m)$, which still satisfies $D\tilde{u} = P_0 \abs{D\tilde{u}}$, and we can apply the above reasoning to that function. Since $C'$ was arbitrary, we conclude the proof.
\end{proof}

\begin{remark}
Statement (ii) can also be proved in a slightly different, less elementary fashion using the theory of one-dimensional sections of $\BV$ functions (see e.g.~\cite[Section~3.11]{AmFuPa00FBVF}). Let $P = a \otimes \xi$ and pick any $b \perp \xi$. Then, the theory of sections implies that (the scalar product here is to be taken row-wise)
\[
 Du \cdot b = \Lcal^{d-1} \restrict \Omega_b \otimes Du_y^b = \int_{\Omega_b} Du_y^b \dd y,
\]
where $\Omega_b$ is the orthogonal projection of $\Omega$ onto the hyperplace orthogonal to $b$ and $u_y^b(t) := u(y + tb)$ for any $y \in \Omega_b$ and $t \in \R$ such that $y + tb \in \Omega$. Also,
\[
 \absb{Du \cdot b} = \Lcal^{d-1} \restrict \Omega_b \otimes \absb{Du_y^b}.
\]
For $Du = P \abs{Du}$ we have $\abs{Du \cdot b} = \abs{a \xi^T b} \abs{Du} = 0$ and hence $\abs{Du_y^b} = 0$ for almost every $y \in \Omega_b$. But this implies that $u$ is constant in direction $b$. As $b \perp \xi$ was arbitrary, $u(x)$ can only depend on $x \cdot \xi$ and we have shown the claim.
\end{remark}

\begin{remark}[Differential inclusions] \label{rem:rigidity}
Restating the preceding lemma, we have proved rigidity for the differential inclusion
\begin{equation} \label{eq:diff_incl}
  Du \in \mathrm{span} \{P\},  \qquad u \in \BV(C;\R^m),
\end{equation}
which, if we additionally assume $\abs{P} = 1$, is to be interpreted as
\[
  \frac{\di Du}{\di \abs{Du}}(x) = P \qquad\text{for $\abs{Du}$-a.e.\ $x \in C$.}
\]
Notice that for $u \in \Wrm^{1,1}(C;\R^m)$, this simply means
\[
  \nabla u(x) \in \mathrm{span} \{P\}, \qquad\text{for a.e.\ $x \in C$.}
\]
Rigidity here refers to the fact that~\eqref{eq:diff_incl} has either only affine solutions if $\rk P \geq 2$, or one-directional solutions (\enquote{plane waves}) in direction $\xi$ if $P = a \otimes \xi$. For the terminology also cf.\ Definition~1.1 in~\cite{Kirc03RGM}.
\end{remark}

The following corollary is not needed in the sequel, but is included for completeness.

\begin{corollary}[Rigidity for approximate solutions]
If $P \in \R^{m \times d}$ with $\rk P \geq 2$, and $(u_j) \subset \Wrm^{1,\infty}(\Omega;\R^m)$ is a sequence such that
\begin{align*}
  u_j &\toweakstar u \quad\text{in $\Wrm^{1,\infty}(\Omega;\R^m)$}  \qquad\text{and} \\
  \dist \bigl( \nabla u_j,\mathrm{span} \{P\} \bigr) &\to 0 \quad\text{in measure,}
\end{align*}
then even
\[
  \nabla u_j \to \const  \quad\text{in measure.}
\]
\end{corollary}

\begin{proof}
This follows from Lemma~2.7~(ii) in~\cite{Mull99VMMP} and the Rigidity Lemma.
\end{proof}

The following lemma applies to all types of blow-ups and is essentially a local version of the Rigidity Lemma~\ref{lem:BV_rigidity}.

\begin{lemma}[Local structure of $BV$-derivatives] \label{lem:BV_local_structure}
Let $u \in \BV(\Omega;\R^m)$ and let $C \subset \R^d$ be an open convex set containing the origin. Then, for every $x_0 \in \supp Du$, each $\tau \in \Tan_C(Du,x_0)$ is a $\BV$-derivative, $\tau = Dv$ for some $v \in \BV(C;\R^m)$, and with $P_0 := \frac{\di Du}{\di \abs{Du}}(x_0)$ (assume that this exists as a limit) it holds that:
\begin{itemize}
  \item[(i)]  If $\rk P_0 \geq 2$, then $v(x) = v_0 + \alpha P_0 x$ (a.e.), where $\alpha \in \R$, $v_0 \in \R^m$.
  \item[(ii)] If $P_0 = a \otimes \xi$ ($a \in \R^m$, $\xi \in \Sbb^{d-1}$), then there exist $\psi \in \BV(\R)$, $v_0 \in \R^m$ such that $v(x) = v_0 + \psi(x \cdot \xi)a$ (a.e.).
\end{itemize}
\end{lemma}

Notice that we can indeed treat \emph{all} $x_0 \in \supp Du$.

\begin{proof}
Assume $\gamma_n := c_n T^{(x_0,r_n)}_* Du \toweakstar \tau$ in $\Mbf(C;\R^m)$, where as usual we have set $c_n := \abs{Du}(C(x_0,r_n))^{-1}$ (recall $C(x_0,r_n) := x_0 + r_n C$). Let $v_n \in \BV(C;\R^m)$ be defined by
\[
  v_n(y) := \frac{r_n^{d-1}}{\abs{Du}(C(x_0,r_n))} \bigl( u(x_0 + r_n y) - \bar{u}^{(r_n)} \bigr),
  \qquad y \in C,
\]
where $\bar{u}^{(r_n)} = \dashint_{C(x_0,r_n)} u \dd x$. Integration by parts yields
\[
  Dv_n = \frac{Du(x_0 + r_n \frarg)}{\abs{Du}(C(x_0,r_n))} = c_n T^{(x_0,r_n)}_* Du = \gamma_n.
\]
By the Poincar\'{e} inequality, the sequence $(v_n)$ is uniformly bounded in $\BV(C;\R^m)$ and hence $v_n \toweakstar v$ in $\BV(C;\R^m)$ with $Dv = \tau$ (even without selecting a subsequence since the limit is unique). Moreover, from~\eqref{eq:Tan_density} we get that $Dv = P_0 \abs{Dv}$. Hence, we are in the situation of Lemma~\ref{lem:BV_rigidity} and the conclusion follows from this.
\end{proof}

\begin{remark}[Comparison to Alberti's Rank-One Theorem]
The preceding lemma can be seen as a weaker version of Alberti's Rank-One Theorem~\cite{Albe93ROPD}, which asserts that
\[
  P_0 := \frac{\di D^s u}{\di \abs{D^s u}}(x_0) \in \setb{ a \otimes \xi }{ a \in \R^m, \; \xi \in \Sbb^{d-1} }
\]
for $\abs{D^s u}$-almost every $x_0 \in \R^d$. From the Local Structure Lemma we get that every tangent measure $\tau \in \Tan_C(Du,x_0)$ at almost every point $x_0 \in \supp Du$ is the derivative of a $\BV$-function $v \in \BV(C;\R^m)$, which either has the form $v(x) = v_0 + \alpha P_0 x$ or $v(x) = v_0 + \psi(x \cdot \xi)a$ with $a \in \R^m$, $\xi \in \Sbb^{d-1}$, $\psi \in \BV(\R)$, and $v_0 \in \R^m$. In the latter case also $P_0 = a \otimes \xi$, and only at these points we can assert that the conclusion of Alberti's Theorem holds. But Preiss's Example 5.9(1) from~\cite{Prei87GMDR} shows that the first case may even occur almost everywhere, so the result is potentially much weaker than Alberti's. Nevertheless, our lemma still asserts that locally at singular points, $\tau = Dv$ is always one-directional, i.e.\ translation-invariant in all but at most one direction (which usually is proved as a corollary to Alberti's Theorem) and this will suffice later on. On a related note, Preiss's example alluded to above is also the reason why Alberti's Theorem cannot be proved by a blow-up argument, see Section~3 of~\cite{DeLe09NARO} for further explanation.
\end{remark}

\section{Tangent Young measures and localization}   \label{sc:localization}

\subsection{Localization at regular points} \label{ssc:localization}

We first investigate blow-ups of gradient Young measures at regular points.

\begin{proposition}[Localization at regular points] \label{prop:localize_reg}
Let $\nu \in \GYbf(\Omega;\R^{m \times d})$ be a gradient Young measure. Then, for $\Lcal^d$-almost every $x_0 \in \Omega$ there exists a \term{regular tangent Young measure} $\sigma \in \GYbf(\Bbb^d;\R^{m \times d})$ to $\nu$ at $x_0$, that is
\begin{equation} \label{eq:loc_reg_sigma}
  \sigma_y = \nu_{x_0} \quad\text{a.e.,} \qquad
  \lambda_\sigma = \frac{\di \lambda_\nu}{\di \Lcal^d}(x_0) \, \Lcal^d \restrict \Bbb^d, \qquad
  \sigma_y^\infty = \nu_{x_0}^\infty \quad\text{a.e.,}
\end{equation}
in particular
\begin{align}
  [\sigma](\Bbb^d) &= \biggl[ \dprb{\id,\nu_{x_0}} + \dprb{\id,\nu_{x_0}^\infty}
    \frac{\di \lambda_\nu}{\di \Lcal^d}(x_0) \biggr] \abs{\Bbb^d},
    \label{eq:loc_reg_barycenter} \\
  \ddprb{\ONE \otimes h, \sigma} &= \biggl[ \dprb{h,\nu_{x_0}} + \dprb{h^\infty,\nu_{x_0}^\infty}
    \frac{\di \lambda_\nu}{\di \Lcal^d}(x_0) \biggr] \abs{\Bbb^d} \label{eq:loc_reg_ddpr}
\end{align}
for all $\ONE \otimes h \in \Ebf(\Bbb^d;\R^{m \times d})$.
\end{proposition}

\begin{proof}
Let $\{f_k\} = \setn{ \phi_k \otimes h_k \in \Crm(\cl{\Omega}) \times \Crm(\R^{m \times d}) }{ k \in \N } \subset \Ebf(\Omega;\R^{m \times d})$ be as in Lemma~\ref{lem:E_separable}. Further, assume that $x_0 \in \Omega$ is such that
\begin{equation} \label{eq:x0_reg_point}
  \lim_{r \todown 0} \frac{\lambda_\nu^s(B(x_0,r))}{r^d} = 0
\end{equation}
and $x_0$ is a Lebesgue point of the functions
\[
  x \mapsto \dprb{h_k,\nu_x} + \dprb{h_k^\infty,\nu_x^\infty}
  \frac{\di \lambda_\nu}{\di \Lcal^d}(x)
  \qquad\text{for all $k \in \N$.}
\]
By standard results in measure theory, $\Lcal^d$-almost every $x_0 \in \Omega$ satisfies the above hypotheses. 

Take $\nabla u_j \toY \nu$, where $(u_j) \subset \Wrm^{1,1}(\Omega;\R^d)$ (cf.\ Section~\ref{ssc:GYM} for the possibility of finding such a sequence). For $r > 0$ such that $B(x_0,r) \subset \Omega$ introduce $v_j^{(r)} \in \BV(\Bbb^d;\R^m)$,
\[
  v_j^{(r)}(y) := \frac{u_j(x_0 + ry) - \bar{u}_j^{(r)}}{r},  \qquad y \in \Bbb^d,
\]
where $\bar{u}_j^{(r)} = \dashint_{B(x_0,r)} u_j \dd x$. We get $\nabla v_j^{(r)}(y) = \nabla u_j(x_0 + ry)$ and the Poincar\'{e} inequality yields a uniform norm-bound on the sequence $(v_j^{(r)})_j$ in $\Wrm^{1,1}(\Bbb^d;\R^m)$ for fixed $r > 0$. Hence, for every $r > 0$ as above we may pick a subsequence of $j$s (depending on $r$, which is implicit in the following) such that $Dv_j^{(r)} \toY \sigma^{(r)} \in \Ybf(\Bbb^d;\R^{m \times d})$.

Let $\phi \otimes h \in \Ebf(\Bbb^d;\R^{m \times d})$ and use a change of variables to see
\begin{align*}
  \ddprb{\phi \otimes h,\sigma^{(r)}} &= \lim_{j\to\infty} \int_{\Bbb^d} \phi(y) h \bigl(
    \nabla v_j^{(r)}(y) \bigr) \dd y \\
  &= \lim_{j\to\infty} \int_{\Bbb^d} \phi(y) h \bigl(\nabla u_j(x_0 + r y) \bigr) \dd y \\
  &= \lim_{j\to\infty} \frac{1}{r^d} \int_{B(x_0,r)} \phi \Bigl( \frac{x-x_0}{r} \Bigr)
    h \bigl(\nabla u_j(x) \bigr) \dd x \\
  &= \frac{1}{r^d} \, \ddprB{\phi\Bigl(\frac{\frarg-x_0}{r}\Bigr) \otimes h,\nu} .
\end{align*}
We also have
\begin{align*}
  &\frac{1}{r^d} \int_{B(x_0,r)} \phi\Bigl(\frac{x-x_0}{r}\Bigr) \biggl[ \dprb{h,\nu_x}
    + \dprb{h^\infty,\nu_x^\infty} \frac{\di \lambda_\nu}{\di \Lcal^d}(x) \biggr] \dd x \\
  &\qquad = \int_{\Bbb^d} \phi(y) \biggl[ \dprb{h,\nu_{x_0+r y}}
    + \dprb{h^\infty,\nu_{x_0+r y}^\infty} \frac{\di \lambda_\nu}{\di \Lcal^d}(x_0 + r y)
    \biggr] \dd y,
\end{align*}
which by assumption on $x_0$ converges to
\[
  \int_{\Bbb^d} \phi(y) \biggl[ \dprb{h,\nu_{x_0}} + \dprb{h^\infty,\nu_{x_0}^\infty}
    \frac{\di \lambda_\nu}{\di \Lcal^d}(x_0) \biggr] \dd y
\]
as $r \todown 0$. To see this, first consider the collection $\{ f_k \} = \{ \phi_k \otimes h_k \}$ from above, employ the Lebesgue point properties of $x_0$, and then use Lemma~\ref{lem:E_separable} to get the assertion for all $\phi \otimes h \in \Ebf(\Bbb^d;\R^{m \times d})$. For the singular part, we have
\begin{align*}
  &\absBB{ \frac{1}{r^d} \int_{B(x_0,r)} \phi\Bigl(\frac{x-x_0}{r}\Bigr) \dprb{h^\infty,\nu_x^\infty}
    \dd \lambda_\nu^s(x)} \\
  &\qquad \leq \frac{\lambda_\nu^s(B(x_0,r))}{r^d} \norm{\phi}_\infty \cdot
    \sup_{A \in \partial \Bbb^{m \times d}} \, \absb{h^\infty(A)}
    \quad\to\quad 0 \qquad\text{as $r \todown 0$}
\end{align*}
by~\eqref{eq:x0_reg_point}. In particular, we have proved
\[
  \limsup_{r \todown 0} \, \absb{\ddprb{\phi \otimes h,\sigma^{(r)}}} < \infty,
\]
and so by virtue of Theorem~\ref{thm:YM_compactness} we may choose a sequence $r_n \todown 0$ such that $\sigma^{(r_n)} \toweakstar \sigma$ in $\Ybf(\Bbb^d;\R^{m \times d})$. A diagonal argument yields $\sigma \in \GYbf(\Bbb^d;\R^{m \times d})$.

Hence,
\begin{align*}
  \ddprb{\phi \otimes h,\sigma} &= \lim_{n\to\infty} \ddprb{\phi \otimes h,\sigma^{(r_n)}} \\
  &= \int_{\Bbb^d} \phi(y) \biggl[ \dprb{h,\nu_{x_0}} + \dprb{h^\infty,\nu_{x_0}^\infty}
    \frac{\di \lambda_\nu}{\di \Lcal^d}(x_0) \biggr] \dd y.
\end{align*}
Varying $\phi$ and then $h$, we get~\eqref{eq:loc_reg_sigma}, from which the assertions~\eqref{eq:loc_reg_barycenter},~\eqref{eq:loc_reg_ddpr} follow immediately.
\end{proof}

\subsection{Localization at singular points}

We now consider singular points.

\begin{proposition}[Localization at singular points] \label{prop:localize_sing}
Let $\nu \in \GYbf(\Omega;\R^{m \times d})$ be a gradient Young measure. Then, there exists a set $S \subset \Omega$ with $\lambda_\nu^s(\Omega \setminus S) = 0$ such that for all $x_0 \in S$ and all open cubes $Q$ with center $0$ and $\abs{Q} = 1$, there exists a \term{singular tangent Young measure} $\sigma \in \GYbf(Q;\R^{m \times d})$ to $\nu$ at $x_0$, that is
\begin{align}
  [\sigma] &\in \absn{\dprn{\id,\nu_{x_0}^\infty}} \cdot \Tan_Q([\nu],x_0),  &\sigma_y &= \delta_0 \quad\text{a.e.,} \label{eq:loc_sing_1} \\
  \lambda_\sigma &\in \Tan_Q(\lambda_\nu^s,x_0), \quad \lambda_\sigma(Q) = 1,  &\sigma_y^\infty &= \nu_{x_0}^\infty
    \quad\text{$\lambda_\sigma$-a.e.} \label{eq:loc_sing_2}
\end{align}
In particular, for all bounded open sets $U \subset Q$ with $(\Lcal^d + \lambda_\sigma)(\partial U) = 0$ and all positively $1$-homogeneous $g \in \Crm(\R^{m \times d})$ it holds that
\begin{equation} \label{eq:loc_sing_ddpr}
  \ddprb{\ONE_U \otimes g, \sigma} = \dprb{g,\nu_{x_0}^\infty} \, \lambda_\sigma(U).
\end{equation}
\end{proposition}

\begin{proof}
Denote by $Q(x_0,r)$ the cube $x_0 + rQ$ and let $S \subset \supp \lambda_\nu^s \subset \Omega$ be the set of all points $x_0 \in \Omega$ satisfying
\begin{equation} \label{eq:sing_point}
  \lim_{r \todown 0} \frac{1}{\lambda_\nu^s(Q(x_0,r))} \int_{Q(x_0,r)} 1 +
  \dprb{\abs{\frarg},\nu_x} + \frac{\di \lambda_\nu}{\di \Lcal^d}(x) \dd x = 0
\end{equation}
and such that $x_0$ is a $\lambda_\nu^s$-Lebesgue point of the functions
\[
  x \mapsto \dprb{\id,\nu_x^\infty}  \qquad\text{and}\qquad
  x \mapsto \dprb{g_k,\nu_x^\infty}, \quad k \in \N,
\]
where $\setn{g_k}{ k \in \N }$ is a countable dense family of functions in $\Crm(\partial \Bbb^{m \times d})$.
From standard results in measure theory, in particular the strong form of Besicovitch's Derivation Theorem (Theorem~1.153 in~\cite{FonLeo07MMCV}), both conditions hold for $\lambda_\nu^s$-almost every $x_0 \in \Omega$, hence $\lambda_\nu^s(\Omega \setminus S) = 0$, and $S$ can be chosen independently of the cube $Q$.

Moreover, possibly discarding another $\lambda_\nu^s$-negligible subset of $S$ (independent of the cube $Q$), at every fixed $x_0 \in S$ the Strict Blow-up Lemma~\ref{lem:strict_blowup} gives a sequence $r_n \todown 0$ such that
\begin{equation} \label{eq:lambda_blowup}
  \frac{T_*^{(x_0,r_n)} \lambda_\nu^s}{\lambda_\nu^s(Q(x_0,r_n))}
  \quad\to\quad \lambda \in \Tan_Q(\lambda_\nu^s,x_0)  \qquad\text{strictly}
\end{equation}
as well as
\[
  \lambda(Q) = 1  \qquad\text{and}\qquad
  \lambda_\nu(\partial Q(x_0,r_n)) = 0  \quad\text{for all $n \in \N$.}
\]
Let $\nabla u_j \toY \nu$, where $(u_j) \subset \Wrm^{1,1}(\Omega;\R^m)$. For all $n \in \N$ large enough so that $Q(x_0,r_n) \subset \Omega$, let the functions $v_j^{(r_n)} \in \BV(Q(x_0,r_n);\R^m)$ be defined by
\[
  v_j^{(r_n)}(y) := r_n^{d-1}c_n \bigl( u_j(x_0 + r_n y) - \bar{u}_j^{(r_n)} \bigr),  \qquad y \in Q,
\]
where $\bar{u}_j^{(r_n)} = \dashint_{Q(x_0,r_n)} u_j \dd x$ and
\[
  c_n := \frac{1}{\ddprn{\ONE_{Q(x_0,r_n)} \otimes \abs{\frarg},\nu}}.
\]
Notice that $c_n$ is well-defined, because $x_0 \in \supp \lambda_\nu$.

We have
\begin{align*}
  \nabla v_j^{(r_n)}(y) &= r_n^d c_n \nabla u_j(x_0 + r_n y), \qquad y \in Q, \\
  Dv_j^{(r_n)} &= c_n T_*^{(x_0,r_n)} Du_j.
\end{align*}
Moreover,
\[
  \lim_{j\to\infty} c_n \abs{Du_j}(Q(x_0,r_n)) = \lim_{j\to\infty} \frac{\abs{Du_j}(Q(x_0,r_n))}{\ddprn{\ONE_{Q(x_0,r_n)} \otimes \abs{\frarg},\nu}} = 1,
\]
because $\lambda_\nu(\partial Q(x_0,r_n)) = 0$ in conjunction with part~(ii) of Proposition~\ref{prop:ext_repr}. Hence, also using the Poincar\'{e} inequality in BV, the sequence $(v_j^{(r_n)})_j$ is uniformly bounded in the space $\BV(Q;\R^m)$ for every fixed $r_n$ as above. Select a subsequence of the $j$s (depending on $n$, not relabeled) with $Dv_j^{(r_n)} \toY \sigma^{(r_n)} \in \Ybf(Q;\R^{m \times d})$.

By construction, the barycenters $[\sigma^{(r_n)}]$ satisfy
\begin{equation} \label{eq:sigma_rn_barycenter}
  [\sigma^{(r_n)}] = \wslim_{j\to\infty} c_n T_*^{(x_0,r_n)} Du_j = c_n T_*^{(x_0,r_n)} [\nu],
\end{equation}
where the weak* limit is to be understood in $\Mbf(\cl{Q};\R^{m \times d})$.

For every positively $1$-homogeneous $g \in \Crm(\R^{m \times d})$, perform a change of variables to observe that for all $\phi \in \Crm(\cl{Q})$ and $n \in \N$,
\begin{align*}
  \ddprb{\phi \otimes g, \sigma^{(r_n)}} &= \lim_{j \to \infty} \int_Q \phi(y) g \bigl( \nabla v_j^{(r_n)}(y) \bigr) \dd y \\
  &= \lim_{j \to \infty} r_n^d c_n \int_Q \phi(y) g \bigl( \nabla u_j(x_0 + r_n y) \bigr) \dd y \\
  &= \lim_{j \to \infty} c_n \int_{Q(x_0,r_n)} \phi \Big( \frac{x-x_0}{r_n} \Big) g \bigl( \nabla u_j(x) \bigr) \dd x \\
  &= c_n \ddprB{\phi\Bigl(\frac{\frarg-x_0}{r_n}\Bigr) \otimes g,\nu},
\end{align*}
the last equality here follows since $\lambda_\nu(\partial Q(x_0,r_n)) = 0$ as above by part~(ii) of Proposition~\ref{prop:ext_repr}.

In particular,
\[
  \ddprb{\ONE_{\cl{Q}} \otimes \abs{\frarg}, \sigma^{(r_n)}} = 1  \qquad\text{for all $n \in \N$.}
\]
Therefore, up to a subsequence of $n$ (not relabeled), we can assume $\sigma^{(r_n)} \toweakstar \sigma \in \Ybf(Q;\R^{m \times d})$ and a diagonal argument yields $\sigma \in \GYbf(Q;\R^{m \times d})$.

Setting $M := \sup_{A \in \partial \Bbb^{m \times d}} \, \abs{g(A)}$, we can estimate the regular part of $\ddprn{\phi \otimes g, \sigma^{(r_n)}}$ as follows:
\begin{align}
  &\absBBB{ \frac{1}{\ddprn{\ONE_{Q(x_0,r_n)} \otimes \abs{\frarg},\nu}}
    \int_{Q(x_0,r_n)}\phi \Bigl( \frac{x-x_0}{r_n} \Bigr)
    \biggl[ \dprb{g,\nu_x} + \dprb{g,\nu_x^\infty}
    \frac{\di \lambda_\nu}{\di \Lcal^d}(x) \biggr] \dd x } \notag \\
  &\qquad \leq \frac{\norm{\phi}_\infty M}{\lambda_\nu^s(Q(x_0,r_n))}
    \int_{Q(x_0,r)} \dprb{\abs{\frarg},\nu_x} + \frac{\di \lambda_\nu}
    {\di \Lcal^d}(x) \dd x, \label{eq:reg_part_vanish}
\end{align}
and this goes to zero as $n \to \infty$ by means of~\eqref{eq:sing_point}. Thus, also noticing
\[
  \lim_{r \todown 0} c_n \lambda_\nu^s(Q(x_0,r)) = \lim_{r \todown 0} \frac{\lambda_\nu^s(Q(x_0,r))}{\ddprn{\ONE_{Q(x_0,r)}
  \otimes \abs{\frarg},\nu}} = 1
\]
by the assumptions on $x_0 \in S$, it follows that
\begin{equation} \label{eq:sigma_limit}
\begin{aligned}
  &\ddprb{\phi \otimes g, \sigma} = \lim_{n\to\infty}
    \ddprb{\phi \otimes g, \sigma^{(r_n)}} \\
  &\qquad = \lim_{n\to\infty} \frac{1}{\lambda_\nu^s(Q(x_0,r_n))}
    \int_{Q(x_0,r_n)} \phi \Big( \frac{x-x_0}{r_n} \Big)
    \dprb{g,\nu_x^\infty} \dd \lambda_\nu^s(x).
\end{aligned}
\end{equation}

In~\eqref{eq:loc_sing_1}, the first assertion follows from~\eqref{eq:sigma_rn_barycenter} together with
\[
  \lim_{n\to\infty} c_n \abs{[\nu]}(Q(x_0,r_n)) = \absb{\dprb{\id,\nu_{x_0}^\infty}}
\]
by the assumed properties of $x_0$. 

For the second assertion in~\eqref{eq:loc_sing_1}, consider cut-off functions $\phi \in \Crm_c(Q;[0,1])$, $\chi \in \Crm_c(\R^{m \times d};[0,1])$, and derive similarly to above
\[
  \ddprb{\phi \otimes \abs{\frarg}\chi(\frarg), \sigma^{(n)}} = c_n \ddprB{\phi\Bigl(\frac{\frarg-x_0}{r_n}\Bigr) \otimes \abs{\frarg}\chi(r_n^d c_n \frarg),\nu}.
\]
Since $\chi$ has compact support, the singular part of the last expression is zero, whereas for the regular part we can use a reasoning analogous to~\eqref{eq:reg_part_vanish} to conclude
\[
  \ddprb{\phi \otimes \abs{\frarg}\chi(\frarg), \sigma} = 0,
\]
and we infer $\sigma_x = \delta_0$ almost everywhere (with respect to $\Lcal^d$).

If we use $g = \abs{\frarg}$ in~\eqref{eq:sigma_limit}, we arrive at
\begin{align*}
  \int_Q \phi \dd \lambda_\sigma &= \lim_{n\to\infty} \frac{1}{\lambda_\nu^s(Q(x_0,r_n))}
    \int_{Q(x_0,r_n)} \phi \Big( \frac{x-x_0}{r_n} \Big) \dd \lambda_\nu^s(x) \\
  &= \lim_{n\to\infty} \int_Q \phi \dd \biggl( \frac{T_*^{(x_0,r_n)}\lambda_\nu^s}
    {\lambda_\nu^s(Q(x_0,r_n))} \biggr).
\end{align*}
In conjunction with~\eqref{eq:lambda_blowup}, this shows $\lambda_\sigma = \lambda \in \Tan_Q(\lambda_\nu^s,x_0)$, i.e.\ the first assertion in~\eqref{eq:loc_sing_2}, and the second statement follows immediately from the assumed properties of $\lambda$.

To show~\eqref{eq:loc_sing_ddpr}, let $U \subset Q$ be a bounded open set with $(\Lcal^d + \lambda_\sigma)(\partial U) = 0$ and take a positively $1$-homogeneous $g_k \in \Crm(\R^{m \times d})$ from the collection exhibited at the beginning of the proof. Then, using $\phi = \ONE_U$ in~\eqref{eq:sigma_limit} (once again using the extended representation result from Proposition~\ref{prop:ext_repr}), we get
\begin{align*}
  \int_U \dprb{g_k,\sigma_y^\infty} \dd \lambda_\sigma(y) &= \ddprb{\ONE_U \otimes g_k, \sigma} \\
  &= \lim_{n\to\infty} \frac{1}{\lambda_\nu^s(Q(x_0,r_n))} \int_{U(x_0,r_n)} \dprb{g_k,\nu_x^\infty} \dd \lambda_\nu^s(x) \\
  &= \lambda_\sigma(U) \lim_{n\to\infty} \dashint_{U(x_0,r_n)} \dprb{g_k,\nu_x^\infty} \dd \lambda_\nu^s(x) \\
  &= \lambda_\sigma(U) \dprb{g_k,\nu_{x_0}^\infty},
\end{align*}
where the third equality follows from the fact that
\[
  \lim_{n\to\infty} \frac{\lambda_\nu^s(U(x_0,r_n))}{\lambda_\nu^s(Q(x_0,r_n))} = \frac{\lambda_\sigma(U)}{\lambda_\sigma(Q)} = \lambda_\sigma(U),
\]
and the fourth equality is due to the Lebesgue point properties of $x_0$. Thus we have~\eqref{eq:loc_sing_ddpr} for $g = g_k$. By density, this assertion then also holds for all positively $1$-homogeneous $g \in \Crm(\R^{m \times d})$, and the third assertion in~\eqref{eq:loc_sing_2} follows immediately from that by varying $U$ and $g$. This finishes the proof.
\end{proof}

\section{Jensen-type inequalities and lower semicontinuity}   \label{sc:Jensen}

\subsection{Jensen-type inequalities}

This section establishes Jensen-type inequalities for gradient Young measures. We proceed separately for the regular and the singular part of the Young measure and employ in particular the localization principles of the previous section and the Rigidity Lemma~\ref{lem:BV_rigidity}.

The proof of the Jensen-type inequality at a regular point is rather straightforward:

\begin{proposition} \label{prop:Jensen_regular}
Let $\nu \in \GYbf(\Omega;\R^{m \times d})$ be a gradient Young measure. Then, for $\Lcal^d$-almost every $x_0 \in \Omega$ it holds that
\[
  h \biggl( \dprb{\id,\nu_{x_0}} + \dprb{\id,\nu_{x_0}^\infty} \frac{\di \lambda_\nu}{\di \Lcal^d}(x_0) \biggr)
  \leq \dprb{h,\nu_{x_0}} + \dprb{h^\#,\nu_{x_0}^\infty} \frac{\di \lambda_\nu}{\di \Lcal^d}(x_0)
\]
for all quasiconvex $h \in \Crm(\R^{m \times d})$ with linear growth at infinity.
\end{proposition}

\begin{proof}
Let $\sigma \in \GYbf(\Bbb^d;\R^{m \times d})$ be the regular tangent Young measure to $\nu$ at a suitable $x_0 \in \Omega$ as in Proposition~\ref{prop:localize_reg}. In particular, $[\sigma] = A_0 \Lcal^d \restrict \Bbb^d$, where
\[
  A_0 = \frac{[\sigma](\Bbb^d)}{\abs{\Bbb^d}} = \dprb{\id,\nu_{x_0}}
  + \dprb{\id,\nu_{x_0}^\infty} \frac{\di \lambda_\nu}{\di \Lcal^d}(x_0)
  \quad \in \R^{m \times d}.
\]
Use Lemma~\ref{lem:boundary_adjust} to get a sequence $(v_n) \subset \Wrm^{1,1}(\Bbb^d;\R^m)$ with $Dv_n \toY \sigma$ and $v_n(x) = A_0 x$ on $\partial \Bbb^d$. Hence, for all quasiconvex $h \in \Crm(\R^{m \times d})$ with $\ONE \otimes h \in \Ebf(\Bbb^d;\R^{m \times d})$, we have
\[
  h(A_0) \leq \dashint_{\Bbb^d} h(\nabla v_n) \dd x  \qquad\text{for all $n \in \N$.}
\]
Now, use Lemma~\ref{lem:integrand_approx} to get a collection $\{\ONE \otimes h_k\} \subset \Ebf(\Omega;\R^{m \times d})$ such that $h_k \todown h$, $h_k^\infty \todown h^\#$ pointwise, and all $h_k$ have uniformly bounded linear growth constants. Then, by~\eqref{eq:loc_reg_ddpr}, for all $k \in \N$
\begin{align*}
  h(A_0) &\leq \limsup_{n\to\infty} \dashint_{\Bbb^d} h(\nabla v_n) \dd x
    \leq \lim_{n\to\infty} \dashint_{\Bbb^d} h_k(\nabla v_n) \dd x \\
  &= \frac{\ddprb{\ONE_{\cl{Q}} \otimes h_k,\sigma}}{\abs{\Bbb^d}}
    = \dprb{h_k,\nu_{x_0}} + \dprb{h_k^\infty,\nu_{x_0}^\infty} \frac{\di \lambda_\nu}{\di \Lcal^d}(x_0).
\end{align*}
Letting $k\to\infty$ together with the monotone convergence theorem proves the claim of the proposition.
\end{proof}

Establishing a Jensen-type inequality for the singular points is more involved, but the basic principle of blowing up around a point $x_0$ and then using quasiconvexity on the blow-up limit remains the same. Additionally, however, we need to include an averaging procedure since tangent Young measures at singular points usually do not have an affine function as underlying deformation. For this, we need the information provided by the Rigidity Lemma~\ref{lem:BV_rigidity}.

\begin{proposition} \label{prop:Jensen_singular}
Let $\nu \in \GYbf(\Omega;\R^{m \times d})$ be a gradient Young measure. Then, for $\lambda_\nu^s$-almost every $x_0 \in \Omega$ it holds that
\[
   g \bigl( \dprb{\id,\nu_{x_0}^\infty} \bigr) \leq \dprb{g,\nu_{x_0}^\infty}
\]
for all quasiconvex and positively $1$-homogeneous functions $g \in \Crm(\R^{m \times d})$.
\end{proposition}

\begin{remark}
Notice that we did not say anything about the validity of a singular Jensen-type inequality at boundary points $x_0 \in \partial \Omega$. This is also not needed in the sequel.
\end{remark}

\begin{proof}
Let $S \subset \Omega$ be as in Proposition~\ref{prop:localize_sing} and fix $x_0 \in S$. Define (which is possible $\lambda_\nu^s$-almost everywhere)
\[
  A_0 := \dprb{\id,\nu_{x_0}^\infty}.
\]
We distinguish two cases:

\proofstep{Case 1: $\rk A_0 \geq 2$ or $A_0 = 0$ (semi-regular blow-up).}\\
Let $\sigma \in \GYbf(Q;\R^{m \times d})$ be a singular tangent Young measure to $\nu$ at $x_0$, whose existence is ascertained by Proposition~\ref{prop:localize_sing}, $Q = (-1/2,1/2)^d$ the unit cube. In this semi-regular case, we can proceed analogously to the regular blow-up in Proposition~\ref{prop:Jensen_regular}: From Lemma~\ref{lem:boundary_adjust} take $(v_n) \subset \Wrm^{1,1}(Q;\R^m)$ with $\nabla v_n \toY \sigma$, $v_n \toweakstar v \in \BV(Q;\R^m)$, and $v_n|_{\partial Q} = v|_{\partial Q}$. Also, $Dv = A_0 \lambda_\sigma$ and so, invoking the Rigidity Lemma~\ref{lem:BV_rigidity}~(i),
\[
  v(x) = A_0 x,  \quad x \in Q, \qquad\text{and}\qquad
  Dv = A_0 \Lcal^d \restrict Q,
\]
where, without loss of generality, we assumed that the constant part of $v$ is zero. The quasiconvexity of $g$ immediately yields
\[
  g(A_0) \leq \dashint_Q g(\nabla v_n) \dd x
\]
for all $n \in \N$. Then,
\[
  g(A_0) \leq \limsup_{n\to\infty} \dashint_Q g(\nabla v_n) \dd x
  = \ddprb{\ONE_{\cl{Q}} \otimes g,\sigma} = \dprb{g,\nu_{x_0}^\infty},
\]
where the last equality follows from~\eqref{eq:loc_sing_ddpr}.

\proofstep{Case 2: $A_0 = a \otimes \xi$ for $a \in \R^m \setminus \{0\}$, $\xi \in \Sbb^{d-1}$ (fully singular blow-up).}\\ To simplify notation we assume that $\xi = \mathrm{e}_1 = (1,0,\ldots,0)^T$; otherwise the unit cube $Q = (-1/2,1/2)^d$ in the following proof has to be replaced by a rotated unit cube with one face orthogonal to $\xi$.

Like in Case 1, take $\sigma \in \GYbf(Q;\R^{m \times d})$ to be a singular tangent Young measure to $\nu$ at $x_0$ as in Proposition~\ref{prop:localize_sing} and let $(v_n) \subset \Wrm^{1,1}(Q;\R^m)$ with $\nabla v_n \toY \sigma$, $v_n \toweakstar v \in \BV(Q;\R^m)$, and $v_n|_{\partial Q} = v|_{\partial Q}$. Using case (ii) of the Rigidity Lemma~\ref{lem:BV_rigidity}, and adding a constant function if necessary, $v$ can be written in the form
\[
  v(x) = \psi(x_1) a  \qquad\text{for some $\psi \in \BV(-1/2,1/2)$.}
\]
Observe that
\[
  \abs{A_0} = \abs{A_0} \lambda_\sigma(Q) = \abs{[\sigma]}(Q) = \abs{Dv}(Q) = \abs{D\psi}((-1/2,1/2)) \abs{A_0},
\]
and hence
\[
  \abs{D\psi}((-1/2,1/2)) = \psi(1/2-0)-\psi(-1/2+0) = 1.
\]

Set
\[
  \tilde{u}_n(x) := v_n \biggl( x - \floorBB{x + \frac{1}{2}} \biggr) + a \floorBB{x_1 + \frac{1}{2}}, \qquad x \in \R^d,
\]
where $\floor{s}$ is the largest integer smaller than or equal to $s \in \R$ and we have also set $\floor{x} := (\floor{x_1},\ldots,\floor{x_d})$ for $x \in \R^d$. Then define $(u_n) \subset \BV(Q;\R^m)$ by
\[
  u_n = \frac{\tilde{u}_n(nx)}{n},  \qquad x \in Q,
\]
and observe that
\[
  \nabla u_n(x) = \sum_{z \in \{ 0,\ldots,n-1 \}^d} \nabla v_n(nx-z)
    \ONE_{Q(z/n,1/n)}(x)
\]
Furthermore, $Du_n$ has zero singular part, because the gluing discontinuities over the hyperplanes $\{x_1 = \ell/n + 1/(2n)\}$, are compensated by the jumps of the staircase term in the definition of $\tilde{u}_n$.

Moreover, $u_n \to A_0 x$ in $\Lrm^1(Q;\R^m)$ by a change of variables and since $\psi$ is bounded. Therefore, using Lemma~\ref{lem:boundary_adjust} again, we can find another sequence $(w_n) \subset \Wrm^{1,1}(Q;\R^m)$ with $w_n(x) = A_0 x$ for all $x \in \partial Q$, and such that the sequences $(\nabla u_n)$ and $(\nabla w_n)$ generate the same (unnamed) Young measure, in particular
\[
  \lim_{n\to\infty} \int_Q g(\nabla w_n) \dd x
  = \lim_{n\to\infty} \int_Q g(\nabla u_n) \dd x
\]
for all $g$ as in the statement of the proposition. Thus, we get by quasiconvexity
\[
  g(A_0) \leq \dashint_{Q} g(\nabla w_n) \dd x
\]
for all $n \in \N$. Then use~\eqref{eq:loc_sing_ddpr} to deduce
\begin{align*}
  g(A_0) &\leq \limsup_{n\to\infty} \dashint_Q g(\nabla w_n) \dd x
    = \lim_{n\to\infty} \dashint_Q g(\nabla u_n) \dd x \\
  &= \lim_{n\to\infty} \sum_{z \in \{ 0,\ldots,n-1 \}^d} \int_{Q(z/n,1/n)}
    g\bigl( \nabla v_n(nx-z) \bigr) \dd x \\
  &= \lim_{n\to\infty} \sum_{z \in \{ 0,\ldots,n-1 \}^d} \frac{1}{n^d}
    \int_Q g(\nabla v_n) \dd y \\
  &= \lim_{n\to\infty} \int_Q g(\nabla v_n) \dd y
    = \ddprb{\ONE_{\cl{Q}} \otimes g,\sigma}
    = \dprb{g,\nu_{x_0}^\infty}.
\end{align*}
This concludes the proof.
\end{proof}

\subsection{Necessary conditions for gradient Young measures and lower semicontinuity}

The following theorem exhibits necessary conditions for a Young measure to be a gradient Young measure.

\begin{theorem} \label{thm:GYM_necessary}
Let $\nu \in \GYbf(\Omega;\R^{m \times d})$ be a gradient Young measure. Then, for all quasiconvex $h \in \Crm(\R^{m \times d})$ with linear growth at infinity, it holds that
\[
  h \biggl( \dprb{\id,\nu_x} + \dprb{\id,\nu_x^\infty} \frac{\di \lambda_\nu}{\di \Lcal^d}(x) \biggr)
    \leq \dprb{h,\nu_x} + \dprb{h^\#,\nu_x^\infty} \frac{\di \lambda_\nu}{\di \Lcal^d}(x)
\]
for $\Lcal^d$-almost every $x \in \Omega$, and
\[
  h^\# \bigl( \dprb{\id,\nu_x^\infty} \bigr) \leq \dprb{h^\#,\nu_x^\infty}
\]
for $\lambda_\nu^s$-almost every $x \in \Omega$.
\end{theorem}

The proof is contained in Propositions~\ref{prop:Jensen_regular} and~\ref{prop:Jensen_singular} once we notice that if $h$ is quasiconvex, then its generalized recession function $h^\#$ is quasiconvex as well (by Fatou's lemma), and hence continuous (see Section~\ref{ssc:qc}).

In the situation of the theorem there exists $u \in \BV(\Omega;\R^m)$ such that $Du = [\nu] \restrict \Omega$ and the above conditions become
\[
  h(\nabla u(x)) \leq \dprb{h, \nu_x} + \dprb{h^\#, \nu_x^\infty} \,
  \frac{\di \lambda_\nu}{\di \Lcal^d}(x)
  \qquad\text{for $\Lcal^d$-almost every $x \in \Omega$,}
\]
and
\[
  h^\# \Big( \frac{D^s u}{\abs{D^s u}} \Big) \, \abs{D^s u}
  \leq \dprb{ h^\#, \nu_x^\infty} \, \lambda_\nu^s \restrict \Omega
  \qquad\text{as measures,}
\]
for all upper semicontinuous and quasiconvex $h \colon \R^{m \times d} \to \R$ with linear growth at infinity.

We can now prove the main lower semicontinuity result:

\begin{theorem}[Lower semicontinuity in $\BV$] \label{thm:BV_lsc}
Let $\Omega \subset \R^d$ be a bounded Lipschitz domain with boundary unit inner normal $n_{\Omega} \colon \partial \Omega \to \Sbb^{d-1}$ and let $f \colon \cl{\Omega} \times \R^{m \times d} \to \R$ be a Carath\'{e}odory integrand with linear growth at infinity that is quasiconvex in its second argument and for which the recession function $f^\infty$ exists in the sense of~\eqref{eq:f_infty} and is (jointly) continuous. Then, the functional
\begin{align*}
  \Fcal(u) := &\int_\Omega f( x, \nabla u) \dd x
    + \int_\Omega f^\infty \Bigl( x, \frac{\di D^s u}{\di \abs{D^s u}} \Bigr)
    \dd \abs{D^s u} \\
  &\qquad + \int_{\partial\Omega} f^\infty \bigl( x, u|_{\partial \Omega}
    \otimes n_{\Omega} \bigr) \dd \Hcal^{d-1},
  \qquad u \in \BV(\Omega;\R^m),
\end{align*}
is sequentially lower semicontinuous with respect to weak* convergence in the space $\BV(\Omega;\R^m)$.
\end{theorem}

\begin{proof}
Let $u_j \toweakstar u$ in $\BV(\Omega;\R^m)$. Take a larger Lipschitz domain $\Omega' \supset\supset \Omega$ and consider all $u_j, u$ to be extended to $\Omega'$ by zero. Assume also that $Du_j \toY \nu \in \GYbf(\Omega';\R^{m \times d})$, for which it follows that
\[
  [\nu] = Du \restrict \Omega + (u|_{\partial \Omega} \otimes n_{\Omega}) \, \Hcal^{d-1} \restrict \partial \Omega.
\]
This entails taking a subsequence if necessary, but since we will show an inequality for all such subsequences, it also holds for the original sequence. Observe that if $\lambda_\nu^*$ is the singular part of $\lambda_\nu$ with respect to $\abs{D^s u} + \Hcal^{d-1} \restrict \partial\Omega$, i.e.\ $\lambda_\nu^*$ is concentrated in an $(\abs{D^s u} + \Hcal^{d-1} \restrict \partial\Omega)$-negligible set, then
\begin{align*}
  &\dprb{\id,\nu_x} + \dprb{\id,\nu_x^\infty} \frac{\di \lambda_\nu}{\di \Lcal^d}(x)
    = \frac{\di [\nu]}{\di \Lcal^d}(x) = \begin{cases}
      \nabla u(x)  &\text{for $\Lcal^d$-a.e.\ $x \in \Omega$,} \\
      0            &\text{for $\Lcal^d$-a.e.\ $x \in \Omega' \setminus \Omega$,}
    \end{cases} \\
  &\frac{\dpr{\id,\nu_x^\infty}}{\abs{\dpr{\id,\nu_x^\infty}}}
    = \frac{\di [\nu]^s}{\di \abs{[\nu]^s}}(x)
    = \begin{cases}
      \displaystyle\frac{\di D^s u}{\di \abs{D^s u}}(x)
         &\text{for $\abs{D^su}$-a.e.\ $x \in \Omega$,} \\
      \displaystyle\frac{u|_{\partial \Omega}(x)}{\abs{u|_{\partial \Omega}(x)}} \otimes n_{\Omega}(x)
         &\text{for $\abs{u} \Hcal^{d-1}$-a.e.\ $x \in \partial \Omega$,}
    \end{cases} \\
  &\dprb{\id,\nu_x^\infty} = 0  \qquad\text{for $\lambda_\nu^*$-a.e.\
    $x \in \cl{\Omega'}$,} \\
  &\abs{\dpr{\id,\nu_x^\infty}} \lambda_\nu^s = \abs{D^s u}
    + \abs{u} \Hcal^{d-1} \restrict \partial \Omega, \\
  &\dprb{\id,\nu_x} = 0 \qquad\text{for all $x \in \Omega' \setminus \cl{\Omega}$,}\\
  &\lambda_\nu \restrict (\cl{\Omega'} \setminus \cl{\Omega}) = 0.
\end{align*}
Extend $f$ to $\Omega' \times \R^{m \times d}$ as follows: first extend $f^\infty$ restricted to $\cl{\Omega} \times \partial \Bbb^{m \times d}$ continuously to $\Omega' \times \partial \Bbb^{m \times d}$ and then set $f(x,A) := \abs{A}f^\infty(x,A/\abs{A})$ for $x \in \Omega' \setminus \cl{\Omega}$ and $A \in \R^{m \times d}$. This extended $f$ is still a Carath\'{e}odory function, $f^\infty$ is jointly continuous and $f(x,0) = 0$ for all $x \in \Omega' \setminus \cl{\Omega}$.

Then, from Theorem~\ref{thm:GYM_necessary} (in $\Omega'$) together with the extended representation result for generalized Young measures, Proposition~\ref{prop:ext_repr} in Section~\ref{ssc:YM}, we get
\begin{align*}
  \liminf_{j\to\infty} \Fcal(u_j) &= \int_{\Omega'} \dprb{f(x,\frarg),\nu_x}
    + \dprb{f^\infty(x,\frarg),\nu_x^\infty} \frac{\di \lambda_\nu}{\di \Lcal^d}(x)
    \dd x \\ 
  &\qquad + \int_{\cl{\Omega'}} \dprb{f^\infty(x,\frarg),\nu_x^\infty}
    \dd \lambda_\nu^s(x) \\
  &\geq \int_\Omega f \biggl( x, \dprb{\id,\nu_x} + \dprb{\id,\nu_x^\infty}
    \frac{\di \lambda_\nu}{\di \Lcal^d}(x) \biggr) \dd x \\
  &\qquad + \int_{\cl{\Omega}} f^\infty \bigl( x,
    \dprb{\id,\nu_x^\infty} \bigr) \dd \lambda_\nu^s(x) \\
  &= \Fcal(u).
\end{align*}
This proves the claim.
\end{proof}

\begin{remark}[Recession functions] \label{rem:rec_funct}
In comparison to previously known results, we have to assume that the \enquote{strong} recession function $f^\infty$ exists instead of merely using the upper generalized recession function $f^\#$. This is in fact an unavoidable phenomenon of our proof strategy without Alberti's Rank-One Theorem: It is well-known (see for instance Theorem~2.5~(iii) in~\cite{AliBou97NUIG}) that the natural recession function for lower semicontinuity is the \emph{lower} generalized recession function
\[
  f_\#(x,A) := \liminf_{t \to \infty} \frac{f(tA)}{t},
  \qquad \text{$x \in \cl{\Omega}$, $A \in \R^{m \times d}$.}
\]
Unfortunately, we cannot easily determine whether this function is quasiconvex, so the singular Jensen-type inequality from Prosposition~\ref{prop:Jensen_singular} is not applicable. The usual proof that $f^\#$ (and hence $f^\infty$) is quasiconvex whenever $f$ is, proceeds by virtue of Fatou's lemma, and this method fails for $f_\#$. One can show, however, that if $f_\#$ is known to be quasiconvex, then the lower semicontinuity theorem also holds for $f_\#$ in place of $f^\infty$. By Alberti's Rank-One Theorem we know that this is the same functional, since $f_\#(x,A) = f^\#(x,A)$ if $\rk A \leq 1$ (by the rank-one convexity of $f$). This should be contrasted with the fact that $f_\#$ and $f^\#$ may differ outside the rank-one cone, see~\cite{Mull92QFHD}.
\end{remark}

We also immediately get the following corollary on the functional without the boundary term:

\begin{corollary}
For every quasiconvex $f \in \Ebf(\Omega;\R^{m \times d})$, the functional
\[
  \Fcal(u) := \int_\Omega f( x, \nabla u) \dd x
  + \int_\Omega f^\infty \Bigl( x, \frac{\di D^s u}{\di \abs{D^s u}} \Bigr) \dd \abs{D^s u},
  \qquad u \in \BV(\Omega;\R^m),
\]
is sequentially lower semicontinuous with respect to all weakly*-converging sequences $u_j \toweakstar u$ in $\BV(\Omega;\R^m)$ if at least one of the following conditions is satisfied:
\[
  f \geq 0  \qquad\text{or}\qquad  \text{$\forall j$: $u_j|_{\partial \Omega} = u|_{\partial \Omega}$.}
\]
\end{corollary}

This follows from Theorem~\ref{thm:BV_lsc} since in all of the above cases the boundary term can be neglected. Note that for signed integrands, the above corollary might be \emph{false}, as can be seen from easy counterexamples.

\begin{remark} \label{rem:YM_free_lsc}
We note that it is also possible to show lower semicontinuity of integral functionals in $\BV(\Omega;\R^m)$ (or relaxation theorems) without the use of Young measures. For example, for the functional
\[
  \Fcal(u) := \int_\Omega f( \nabla u) \dd x + \int_\Omega f^\infty \Bigl(\frac{\di D^s u}{\di \abs{D^s u}} \Bigr) \dd \abs{D^s u},
  \qquad u \in \BV(\Omega;\R^m),
\]
where $f \in \Ebf(\Omega;\R^{m \times d})$ takes only non-negative values, does not depend on $x$, and is quasiconvex, one can follow the proof of lower semicontinuity in~\cite{AmbDal92RBVQ} (reproduced in Section~5.5 of~\cite{AmFuPa00FBVF}) almost completely line-by-line, but replacing the use of Alberti's Rank-One Theorem with the Rigidity Lemma~\ref{lem:BV_rigidity} (or the Local Structure Lemma~\ref{lem:BV_local_structure}). Indeed, in the estimate of the singular part from below, Alberti's Theorem is only used to show that the blow-up limit is one-directional, and we can reach that conclusion also by the aforementioned Rigidity Lemma. The other occurence of Alberti's Rank-One Theorem in that proof concerns the fact that we can use the generalized recession function $f^\#$ instead of requiring the existence of the (strong) recession function $f^\infty$. This, however, cannot be avoided, also cf.\ Remark~\ref{rem:rec_funct}. However, while this alternative proof circumvents the framework of Young measures, it uses other technical results instead (like the De Giorgi--Letta Theorem).
\end{remark}


\providecommand{\bysame}{\leavevmode\hbox to3em{\hrulefill}\thinspace}
\providecommand{\MR}{\relax\ifhmode\unskip\space\fi MR }
\providecommand{\MRhref}[2]{%
  \href{http://www.ams.org/mathscinet-getitem?mr=#1}{#2}
}
\providecommand{\href}[2]{#2}

\end{document}